\documentclass[10pt,reqno]{amsart}
\usepackage{geometry}
\usepackage[colorlinks=true, linkcolor=blue, citecolor=magenta, urlcolor=cyan]{hyperref}

\usepackage{amsmath,amssymb,amsthm}
\usepackage{graphicx}
\usepackage{microtype}
\usepackage{xcolor}
\numberwithin{equation}{section}

\newtheorem{theorem}{Theorem}[section]
\newtheorem{lemma}[theorem]{Lemma}
\newtheorem{proposition}[theorem]{Proposition}

\theoremstyle{definition}

\theoremstyle{remark}

\newcommand{\od}{\mathsf{D}}

\newcommand{\dx}{\,\mathrm{d}x}
\newcommand{\dy}{\,\mathrm{d}y}

\usepackage[misc]{ifsym}
\allowdisplaybreaks

\title[Multiplicity of linear elliptic problem with nonlocal
nonlinearities]{Multiplicity and solution construction in linear elliptic problems
with nonlocal nonlinearities}

\author{Chiun-Chang Lee}
\address{ Institute for Computational and Modeling Science, National Tsing Hua University, Hsinchu 300, Taiwan \texttt{lee2@mx.nthu.edu.tw}}

\begin{document}

\begin{abstract} 
This work investigates linear elliptic equations with multiple nonlocal nonlinearities in a bounded domain, which takes the form
\begin{equation*}
-\nabla\cdot(\od\nabla u) +\lambda fu= \eta \sum_{j=1}^kh_j\boldsymbol{\mathsf{N}}_j[u]+h_0.
\end{equation*}
Here $\lambda$ and $\eta$ are positive parameters, and $\boldsymbol{\mathsf{N}}_j[u]$ represents a nonlocal term dependent on the unknown solution~$u$. All coefficients are defined in the context of a wide range of applications. As such equations generally lack a variational structure, a new approach is developed that combines fixed-point arguments with asymptotic techniques. This method establishes the existence and multiplicity of solutions under specific conditions. Of particular interest is the role of the nonlocal nonlinearities, which lead to diverse structures of solutions. To the best of our knowledge, this property is a new finding not typically observed in related nonlocal elliptic problems. 
\end{abstract}

\keywords{Non-variational elliptic problem, Nonlocal nonlinearity, Multiplicity,   Fixed point argument, Asymptotic analysis}

\subjclass[2020]{35A01, 35J15, 35C20}

\maketitle

\section{\bf  Introduction} 
Early work on nonlocal equations can be traced back to a study by Liouville in 1837 (cf. \cite{L1837}), who investigated a parabolic equation with a nonlocal term depending on an integral of the unknown solution. The corresponding stationary solution also satisfies a nonlocal elliptic equation. Since then, the study of nonlocal elliptic equations has evolved significantly. These models are fundamental in many fields of science, {\color{black} as they often arise as stationary solutions of evolution equations with long-range interactions or spatial averaging, and appear in applications ranging from physics and electrochemistry to population dynamics and material systems with memory}~\cite{B1987,BTV1967,BDS1993,Lee2016,Lee2019,MF2013,N2006,PS2018}. The diverse nature of these applications has motivated a rich body of research on the  uniqueness, multiplicity and qualitative properties of solutions to such equations. Nevertheless, to the best of our knowledge, these topics in recent related works are not yet fully settled, even for linear elliptic equations with nonlocal~terms.

This paper aims to contribute to the understanding of these topics by investigating a class of linear elliptic equations involving nonlocal nonlinearities in a bounded smooth  domain~$\Omega \subset \mathbb{R}^m$ with $m \geq 1$. We begin by introducing the following equation with a {\bf single nonlocal term}~$\boldsymbol{\mathsf{N}}[\,\cdot\,]$ and positive parameters $\lambda$ and $\eta$:
\begin{equation}\label{eq-u}
    -\nabla\cdot(\od(x)\nabla u) +\lambda f(x)u=   \eta h(x)\boldsymbol{\mathsf{N}}[u] \quad\text{in}\,\,\Omega,
\end{equation}
 subject to the Robin boundary condition
\begin{equation}\label{rbd}
 u+b(x) \frac{\partial u}{\partial{\vec{\nu}}}  = g(x)\quad\text{on}\,\,\partial\Omega.
\end{equation}
Equation \eqref{eq-u} can be viewed as a linear inhomogeneous equation with a nonlocal coefficient. The solution $u=u_{\lambda,\eta}$ depends on $\lambda$ and $\eta$, but we omit the subscript for simplicity. Here, $\nabla$ and $\nabla\cdot$ denote the gradient and divergence operators in $\mathbb{R}^m$, respectively. 
$\vec{\nu}$ is the unit outward normal vector to $\partial\Omega$. The matrix-valued function~$\od = \left(D_{ij}\right)_{m\times m}$ defined on $\overline{\Omega}$ is symmetric and satisfies a uniform ellipticity condition, with each entry $D_{ij} \in\mathrm{C}^1( \overline{\Omega}; \mathbb{R})$. In Sections~\ref{sub2-1} and \ref{sub2-2}, we first provide a thorough argument to establish the existence of multiple solutions to a special form of \eqref{eq-u}--\eqref{rbd}.

We are particularly interested in elliptic equations with {\bf multiple nonlocal nonlinearities}, as they are crucial for modeling systems where the steady-state behavior is not determined by local information alone, and the influencing factors are related through complex nonlinear relationships. In plasma physics, a single particle's behavior is influenced not only by its immediate surroundings but also by the overall charge distribution throughout the plasma. This is a classic example of a {\bf nonlocal effect}. When simultaneously considering the complex {\bf nonlinear interactions} between multiple species, such as electrons and ions, the resulting equations contain {\bf multiple nonlocal nonlinear terms}. For a more detailed look, the reader can refer to the papers mentioned in the first paragraph and the references therein.

Based on the insights from Sections~\ref{sub2-1}--\ref{sub2-2}, we extend our study to the more general equation $-\nabla\cdot(\od(x)\nabla u) +\lambda f(x)u= \eta \sum_{j=1}^kh_j(x)\boldsymbol{\mathsf{N}}_j[u]+h_0(x)$ (cf. \eqref{eq-u2}) with $k$ multiple nonlocal terms ($k\geq1$). This investigation and its main results are detailed in Section~\ref{sec2.3}.

%Precise assumptions on these variable coefficients will be specified in \eqref{dhb}, following a discussion of the background for \eqref{eq-u}--\eqref{rbd}.
\subsection{Setting of the problem}
Nonlocal elliptic equations inherently present challenges for rigorous mathematical analysis. A key difficulty is that the maximum principle generally fails in the presence of a nonlocal term in~\eqref{eq-u}, in contrast to some local-type elliptic equations where the maximum principle holds under appropriate assumptions~\cite{L2024}. To gain insight into the nonlocal effect, we assume that the {\it local part} of \eqref{eq-u} satisfies the maximum principle. This allows us to proceed with the following assumptions on the coefficients $\od=\left(D_{ij}\right)_{\color{black}{m\times m}}$, $f$, $h$, $b$ and $g$: 
\begin{equation}\label{dhb}
\begin{cases}
   \text{\bf(a).}\,\,(\text{uniform\,\, ellipticity})\,\, D_{ij}\in\text{C}^{1}(\overline{\Omega};\mathbb{R}),\,\,D_{ij}(x)=D_{ji}(x),    \\
 \hspace*{25pt}\sum\limits_{i,j=1}^mD_{ij}(x)z_iz_j\geq\mathfrak{D}|z|^2,\,\,{\forall}x\in\overline{\Omega}\,\,\text{and}\,\,z=(z_1,...,z_m)\in\mathbb{R}^m,\\
    \hspace*{25pt}\text{where}\,\, \mathfrak{D}\,\, \text{is\, a\, positive\, constant};\\
    \text{\bf(b).}\,\,f\in\text{C}^{0}(\overline{\Omega};(0,\infty)),\,\,\,h\in\text{C}^{0}(\overline{\Omega};\mathbb{R})\,\,\text{with}\,\,h\not\equiv0,\\
 \hspace*{28pt}b\in\text{C}^{0}(\partial\Omega;[0,\infty))\,\,\text{and}\,\,g\in\text{C}^{0}(\partial\Omega;\mathbb{R}).
\end{cases}
\end{equation}

Equation~\eqref{eq-u} is a type of elliptic functional differential equations~(cf. \cite{Sk1997}). In particular, equation~\eqref{eq-u} becomes an integro-differential equation when the nonlocal term~$\boldsymbol{\mathsf{N}}[u]$ is represented by an integral involving the unknown solution~$u$ \cite{FS1998,FY2015}. This type of coupling often leads to qualitative behavior that differs significantly from the purely local case. Inspired by the study of integro-differential equations, particularly those of Fredholm type~\cite{CLW2025,W2011}, we specifically focus on two distinct types for the nonlocal term~$\boldsymbol{\mathsf{N}}[u]$:
\begin{equation}\label{N-2t}
 \boldsymbol{\mathsf{N}}[u]=   \begin{cases}
 \displaystyle\int_\Omega w(y)\mathcal{N}(u(y))\dy&(\text{type~I});\vspace{3pt} \\
 \displaystyle\mathcal{N}(\int_\Omega w(y)u(y)\dy)\qquad&(\text{type~II}),
    \end{cases}
\end{equation} 
where $w$ satisfies $\int_\Omega w(y)\dy\neq0$. These two types of nonlocal terms are of particular interest as they introduce global dependencies into equation~\eqref{eq-u}. Their fundamental structural differences lead to a wide range of applications. For instance, a case related to the Type I nonlocal term in \eqref{N-2t} was studied by Allegretto and Barabanova \cite{AB1996}, specifically for nonlocal equations with separable kernels of the form $h(x)\int_{\Omega}w(y)u(y)\,\text{d}y$. In particular, the assumption that $\int_\Omega w(y)\dy\neq0$ is crucial to guarantee a non-trivial nonlocal effect. The setting in \eqref{N-2t} also includes special cases of the average integral, such as
\begin{equation*}
\frac1{|\Omega|}\int_\Omega\mathcal{N}(u(y))\dy\quad\text{and}\quad\mathcal{N}\left(\frac1{|\Omega|}\int_{\Omega}u(y)\dy\right),
\end{equation*}
where $|\Omega|$ is the usual Lebesgue measure of $\Omega$ in $\mathbb{R}^m$. For the subsequent analysis, we require the following assumptions on the functions defined in \eqref{N-2t}:
\begin{itemize}
    \item[\bf(A1).] $\mathcal{N}:\mathbb{R}\to\mathbb{R}$  is locally Lipschitz continuous.
        \item[\bf(A2).] The weight function $w\in\mathrm{L}^p(\Omega)$ for some $p>1$ and its values on $\Omega$ are not restricted to a single sign.  Without loss of generality, we normalize $w$ by setting $\int_\Omega w(y)\dy=1$.
\end{itemize}

It is worth noting that for some special cases, e.g., when $f\equiv h$ and $\lambda\neq\eta$, equation~\eqref{eq-u}--\eqref{rbd} can be transformed into a linear elliptic equation with an integral type Robin boundary condition~\cite{AR1981}. Specifically, if we further assume $\mathcal{N}(s)=s$, one can define an auxiliary function $U=u-\frac\eta\lambda\boldsymbol{\mathsf{N}}[u]$. We then obtain $\boldsymbol{\mathsf{N}}[U]=(1-\frac\eta\lambda)\boldsymbol{\mathsf{N}}[u]$ such that $U$ satisfies the equation~$-\nabla\cdot(\od(x)\nabla U) +\lambda f(x)U=0$, in $\Omega$, with the  nonlocal boundary condition $U+b(x) \frac{\partial U}{\partial{\vec{\nu}}}  = g(x)+\frac\eta{\eta-\lambda}\boldsymbol{\mathsf{N}}[U]$ on $\partial\Omega$. The similar argument can also be found in \cite{CLM2024,L2023,L2025,LM2023}. However, our focus in this paper is on a more general problem, and we will not pursue this specific transformation here.

\subsection{Equations without variational structure}
When the positive parameter $\eta$ vanishes, \eqref{eq-u} is reduced to a linear elliptic equation, which serves as a classical reference case. However, the nonlocal nature of the terms present in \eqref{N-2t} generally means that equation~\eqref{eq-u}--\eqref{rbd} does not have a variational structure. This poses a significant challenge for the use of standard classical tools from the calculus of variations, such as the direct method. Consequently, a different approach is necessary to establish the existence and multiplicity of solutions. In this paper, we overcome this challenge by developing a novel approach that combines fixed-point arguments with asymptotic techniques, which allows us to directly study the non-variational problem.

While the existence of multiple solutions for semilinear elliptic equations has been extensively studied using variational or topological methods, relatively fewer results are available when nonlocal terms are present. In particular, the influence of nonlocality on the multiplicity structure of solutions remains an open and interesting question in many settings. The aim of this paper is to understand under what conditions nonlocal effects may lead to the existence of multiple steady states, and how the structure of the nonlocal coupling influences this phenomenon.

 Note that the nonlocal term $\boldsymbol{\mathsf{N}}[u]$ in \eqref{eq-u} can be viewed as a parameter-like coefficient that depends on the unknown solution $u$. If we replace $\boldsymbol{\mathsf{N}}[u]$ with any given constant, the resulting linear elliptic problem with the boundary condition~\eqref{rbd} admits a unique solution due to the positivity of $f$ and nonnegativity of $b$. A fundamental problem is to determine the conditions under which equation~\eqref{eq-u}--\eqref{rbd} admits a unique solution or multiple solutions. 
 The main goal of this work can be summarized into the following two questions~(Q1) and (Q2):
\begin{itemize}
\item[\bf (Q1).] {\it Under what conditions on $\lambda$ and $\eta$ does the equation admit a unique solution?}
\item[\bf (Q2).] {\it Under what conditions on $\lambda$ and $\eta$ does the equation admit multiple solutions? In this case, does the nature of the nonlocal term $\boldsymbol{\mathsf{N}}[u]$ play a significant role in determining the number of solutions?}
\end{itemize}

We assume that $\mathcal{N}:\mathbb{R}\to\mathbb{R}$ is globally Lipschitz continuous to address question~(Q1) from a theoretical standpoint. Specifically, we show that the solution to \eqref{eq-u}--\eqref{rbd} is unique provided that $\frac\eta\lambda$ is sufficiently small.
This elementary result is stated as follows.
\begin{proposition}[Uniqueness]\label{prop1}
Assume \eqref{dhb}--\eqref{N-2t}, (A2), and that $\mathcal{N}:\mathbb{R}\to\mathbb{R}$ is globally Lipschitz continuous with Lipschitz constant $\mathcal{L}$. Let $|\Omega|$ denote the Lebesgue measure of $\Omega$ in $\mathbb{R}^m$. Then, for positive parameters $\lambda$ and $\eta$ satisfying
\begin{equation}\label{le=}    \frac{\eta}{\lambda}<\frac{|\Omega|^{\frac1p-1}\min_{\overline{\Omega}}f}{\mathcal{L}||w||_{\mathrm{L}^{p}(\Omega)}\max_{\overline{\Omega}}|h|},
\end{equation}
equation~\eqref{eq-u}--\eqref{rbd} has a unique solution $u\in\mathrm{C}^1(\overline{\Omega})\cap\mathrm{C}^2(\Omega)$.

In particular, if $b(x)\equiv0$ on $\partial\Omega$, then  \eqref{le=} can be sharpened to
\begin{equation}\label{=le}
    \frac{\eta}{\lambda}\left(1+\frac{\mathfrak{D}\Lambda_{\Omega}}{\lambda\min_{\overline{\Omega}}f}\right)^{\max\{\frac12,\frac1p\}-1}< \frac{|\Omega|^{\frac1p-1}\min_{\overline{\Omega}}f}{\mathcal{L}||w||_{\mathrm{L}^{p}(\Omega)}\max_{\overline{\Omega}}|h|},
\end{equation}
where $\Lambda_{\Omega}$ denotes the principal eigenvalue of the Dirichlet problem for $-\Delta$ on $\Omega$. 
\end{proposition}
Although the proof of Proposition~\ref{prop1} is based on the standard contraction mapping theorem, to the best of our knowledge, a similar result is not available in the literature. Therefore, a detailed proof is provided in Section~\ref{APX} for the reader's convenience.

The existence and uniqueness of a solution to equation~\eqref{eq-u}--\eqref{rbd} is guaranteed by Proposition~\ref{prop1} when $\mathcal{N}:\mathbb{R}\to\mathbb{R}$ is globally Lipschitz continuous and the parameter condition \eqref{le=} is satisfied. {\color{black} Although} this approach provides a solid foundation, the global Lipschitz assumption is a restrictive condition for many applications, and we observe that a unique solution may still exist even when conditions \eqref{le=} or \eqref{=le} are not met. In our previous work~\cite{CLW2025}, we investigate the uniqueness and asymptotic behavior of solutions for a class of singularly perturbed linear Fredholm integro-differential equations. However, to date, to the best of our knowledge, we have only been able to obtain sufficient conditions for the uniqueness of solutions.

A more challenging and interesting problem is considered when $\mathcal{N}:\mathbb{R}\to\mathbb{R}$ fails to satisfy global Lipschitz continuity. Such a situation often gives rise to various types of solution, including the possibility of multiple solutions. Therefore, we turn our attention to the assumption~(A1) for $\mathcal{N}:\mathbb{R}\to\mathbb{R}$. (Note that (A1) still includes the global case.) We will focus on the question (Q2), and develop a new approach to investigate the existence of multiple solutions and explore the rich solution structures that can emerge. 
\subsection*{The plan of this paper} 
The remainder of this paper is organized as follows. Section~\ref{sec1-1} includes three subsections that present our proposed methodology and main {\color{black} results}. More precisely, in Sections~\ref{sub2-1}--\ref{sub2-2}, we investigate equation~\eqref{eq-u}--\eqref{rbd} with $f=h$. We clearly demonstrate how multiple solutions can arise for this type of nonlocal equation by providing a specific framework that connects the nonlocal nonlinear term to the existence of multiple solutions. Based on this, we proceed in Section~\ref{sec2.3} to consider the general case with multiple nonlocal nonlinear terms, concluding with a statement of our main results: Proposition~\ref{prop-t}, Theorem~\ref{thm0}, and Theorem~\ref{thm2}. Section~\ref{APX} contains the proof of Proposition~\ref{prop1}. The proofs of Proposition~\ref{prop-t} and Theorem~\ref{thm0} are given in Section~\ref{sec-thm0}, while the proof of Theorem~\ref{thm2} is provided in Section~\ref{sec-thm2}. Section~\ref{sec:conclusion} provides the concluding remarks of this work and discusses an outlook for our future research. Finally, the Appendix provides illustrative examples that demonstrate the validity of assumptions (A3) and (A6).

\section{\bf  Methodology and statement of the main result}\label{sec1-1} 
\subsection{The role of $\mathcal{N}$  on the multiplicity: A framework}\label{sub2-1}
We will present our approach systematically, starting with the special case where~$\lambda=\eta$ and $f=h$ in \eqref{eq-u}, which simplifies the equation to
\begin{equation}\label{eq2u}
-\nabla\cdot(\od(x)\nabla u) + \lambda f(x)(u-\boldsymbol{\mathsf{N}}[u] )=0\quad\text{in}\,\,\Omega.
\end{equation}
To analyze the possibility of multiple solutions to equation~\eqref{eq2u} with the boundary condition~\eqref{rbd}, we introduce two auxiliary equations:
\begin{equation}\label{eq-Phi}
    \begin{cases}
-\nabla\cdot(\od(x)\nabla \Phi_{\lambda}) + \lambda f(x)(\Phi_{\lambda}-1)=0 \quad&\text{in}\,\,\Omega,\vspace{3pt}\\ 
\displaystyle\Phi_{\lambda}+b(x) \frac{\partial\Phi_{\lambda}}{\partial{\vec{\nu}}}  = 0\quad&\text{on}\,\,\partial\Omega,
    \end{cases}
\end{equation}
and
\begin{equation}\label{eq-Psi}
    \begin{cases}
-\nabla\cdot(\od(x)\nabla \Psi_{\lambda}) + \lambda f(x)\Psi_{\lambda}=0 \quad&\text{in}\,\,\Omega,\vspace{3pt}\\ 
\displaystyle\Psi_{\lambda}+b(x) \frac{\partial\Psi_{\lambda}}{\partial{\vec{\nu}}}  = g(x)\quad&\text{on}\,\,\partial\Omega.
    \end{cases}
\end{equation}
Under the assumptions in \eqref{dhb}, for $\lambda>0$, both \eqref{eq-Phi} and \eqref{eq-Psi} admit unique solutions which, by the comparison theorem, satisfy
\begin{equation}\label{Phs}    0\leq\Phi_{\lambda}(x)\leq1\quad\text{and}\quad|\Psi_{\lambda}(x)|\leq\max_{\partial\Omega}|g|,\quad\forall x\in\overline{\Omega}.
\end{equation}
Note also that $\Phi_{\lambda}$ and $\Psi_{\lambda}$ are linearly independent. By linearity, a solution to \eqref{eq2u} and the boundary condition \eqref{rbd} can be expressed in the form of 
\begin{equation}\label{Nu}
    u=\boldsymbol{\mathsf{N}}[u]\Phi_{\lambda}+\Psi_{\lambda}.
\end{equation}

 By \eqref{N-2t}, \eqref{Nu} gives an {\bf implicit integral equation} for $u$. Even if $\mathcal{N}:\mathbb{R}\to\mathbb{R}$ is globally Lipschitz continuous, we are interested in the case where its Lipschitz constant fails to satisfy \eqref{le=} with $\lambda=\eta$ and $f\equiv h$. Although $\Phi_{\lambda}$ and $\Psi_{\lambda}$ have good property,  it is challenging to determine even the existence of the nonlocal term $\boldsymbol{\mathsf{N}}[u]$, let alone its uniqueness or multiplicity, as {\bf the properties of the function $\mathcal{N}$ are expected to play a crucial role.}

Our approach for proving the existence of multiple solutions provides a new perspective. We first assume that the algebraic equation $r=\mathcal{N}(r)$ has at least two distinct roots. By subsequently imposing specific conditions on each root, we are able to establish at least two distinct mappings from \eqref{Nu}.

Furthermore, applying asymptotic techniques, each of these mappings is shown to have a unique fixed point in an appropriate complete metric space as $\lambda>0$ is sufficiently large. These distinct fixed points then correspond to multiple solutions of equation~\eqref{eq2u} with the boundary condition~\eqref{rbd}. This approach highlights that the property of $\mathcal{N}:\mathbb{R}\to\mathbb{R}$ may affect the number of solutions to equation~\eqref{eq2u} with the boundary condition~\eqref{rbd}. 
 
 The detailed argument proceeds as follows. Firstly, we define a set $\boldsymbol{\mathrm{S}}_{\star}$ whose elements are the real roots of $r=\mathcal{N}(r)$ that satisfy specific conditions:
\begin{equation}\label{s0n}
\boldsymbol{\mathrm{S}}_{\star}= \left\{r\in\mathbb{R}\left|\,
    \begin{aligned}
        &r=\mathcal{N}(r)\,\,\text{and~there~exists}\,\,\delta_r>0~\text{depending on}~r\,\,\text{such\,\,that}\\
&
        |\mathcal{N}(s_1)-\mathcal{N}(s_2)|\leq\mathcal{L}_r|s_1-s_2|,\,\,
        \forall s_1,s_2\in[r-\delta_r,r+\delta_r].
    \end{aligned}\right.\right\},
\end{equation}
where the constant $\mathcal{L}_r$ satisfies
\begin{equation}\label{rrL}
0<\mathcal{L}_r<\frac{|\Omega|^{\frac1p-1}}{||w||_{\mathrm{L}^{p}(\Omega)}}.
\end{equation}
Here, by convention, we set $\frac1p=0$ if $p=\infty$. The following assumption is required for our subsequent analysis:
\begin{itemize}
    \item[\bf(A3).] The set  $\boldsymbol{\mathrm{S}}_{\star}$ has $n$ distinct elements $r_1<r_2<\cdots<r_n$, where $n\geq2$ and $n$ can be infinite. 
\end{itemize}
We provide some examples in Appendix (Section~\ref{BPX}) to demonstrate the validity of Assumption (A3). This section also includes examples for the special case where the set  $\boldsymbol{\mathrm{S}}_{\star}$ has infinitely many elements, as detailed in (B2)--(B4).

It is crucial to emphasize that not all roots satisfying $r=\mathcal{N}(r)$ belong to the set $\boldsymbol{\mathrm{S}}_{\star}$. Instead, each $r_i\in\boldsymbol{\mathrm{S}}_{\star}$ serves as a {\bf valid candidate} for constructing a solution to equation~\eqref{eq2u} with the boundary condition~\eqref{rbd} via fixed-point theory (cf. Figure~\ref{fig1}). It should be stressed that our fixed-point argument is not applicable to any root $r$ of the equation $r=\mathcal{N}(r)$ that does not belong to the set $\boldsymbol{\mathrm{S}}_{\star}$. The multiplicity of these solutions is then guaranteed by assumption~(A3), which provides the necessary conditions.

\begin{figure}[t] 
    \centering 
\includegraphics[width=0.66\textwidth]{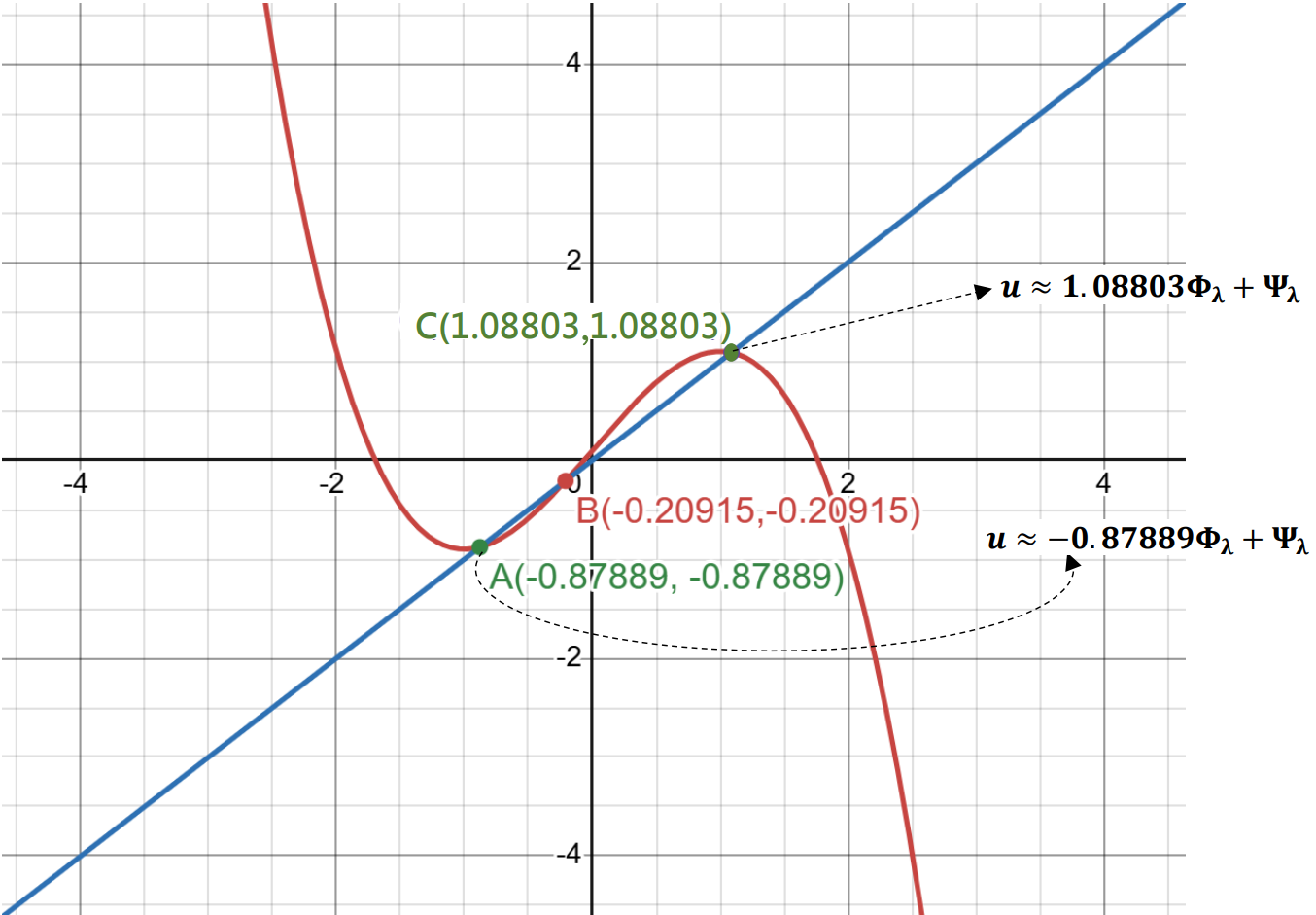} 
\caption{\footnotesize   The figure serves to illustrate that each point in the set $\boldsymbol{\mathrm{S}}_{\star}$ (cf. \eqref{s0n}) corresponds to a solution of the equation~\eqref{eq2u} with the boundary condition \eqref{rbd}. The figure shows the roots of the equation $r=\mathcal{N}(r)$ at three points, A, B, and C, in the coordinate plane, where $\mathcal{N}(r)=\frac{-5r^3+15r+1}{10}$. These roots have approximate values of $r_1\approx-0.87889$, $r_2\approx-0.20915$, and $r_3\approx1.0883$, and we have $|\mathcal{N}'(r_1)|, |\mathcal{N}'(r_3)|<1$ while $|\mathcal{N}'(r_2)|>1$ (also visible in the graph). For simplicity, we consider a case where $w(y)=\frac1{|\Omega|}$. This implies that \eqref{s0n} with  $\mathcal{L}_r\in(0,1)$ defines the set $\boldsymbol{\mathrm{S}}_{\star}=\{r_1,r_3\}$ (corresponding to the green intersection points A and C). By combining fixed-point arguments with asymptotic techniques, our new approach leverages the elements of $\boldsymbol{\mathrm{S}}_{\star}$ to establish a mapping corresponding to \eqref{Nu} and thereby prove the existence of a solution. This enables us to prove the existence of multiple solutions for equation \eqref{eq2u} with the boundary condition \eqref{rbd} when $\lambda$ is sufficiently large. It is important to note, however, that the fixed-point argument is only applicable to roots of the equation $r=\mathcal{N}(r)$ that belong to the set~$\boldsymbol{\mathrm{S}}_{\star}$.} 
    \label{fig1} 
\end{figure}

\subsection{Multiplicity result based on a fixed point argument}\label{sub2-2}
For each $r_i\in \boldsymbol{\mathrm{S}}_{\star}$, $i=1,...,n$, where $\mathcal{N}:\mathbb{R}\to\mathbb{R}$ satisfies (A3), we define a mapping $\boldsymbol{\mathsf{T}}_{i,\lambda}:\mathcal{I}_{i,\lambda}\to\mathcal{I}_{i,\lambda}$ as follows:

\begin{equation}\label{mapT} 
 \boldsymbol{\mathsf{T}}_{i,\lambda}(\theta)= \boldsymbol{\mathsf{N}}[(r_i+\theta)\Phi_{\lambda}+\Psi_{\lambda}]-r_i,\quad\theta\in\mathcal{I}_{i,\lambda},    
 \end{equation}
where $\mathcal{I}_{i,\lambda}\subset\mathbb{R}$ depending on $\lambda$ is a closed set to be determined. The rationale behind this construction is that a fixed point of the mapping $\boldsymbol{\mathsf{T}}_{i,\lambda}:\mathcal{I}_{i,\lambda}\to\mathcal{I}_{i,\lambda}$ (if it exists) corresponds to a solution of equation~\eqref{eq2u} with the boundary condition~\eqref{rbd}. Specifically, if \eqref{mapT} admits a fixed point $\theta_{i,\lambda}$, i.e.,
\begin{equation*}
    \boldsymbol{\mathsf{N}}[(r_i+\theta_{i,\lambda})\Phi_{\lambda}+\Psi_{\lambda}]=r_i+\theta_{i,\lambda}
\end{equation*}
  with $\theta_{i,\lambda}\in\mathcal{I}_{i,\lambda}$, then from \eqref{Nu} we can see that each $u_{i,\lambda}:=(r_i+\theta_{i,\lambda})\Phi_{\lambda}+\Psi_{\lambda}$ is a solution of equation~\eqref{eq2u} with the boundary condition~\eqref{rbd}. 

The main difficulty, therefore, lies in constructing appropriate closed sets $\mathcal{I}_{i,\lambda}\subset\mathbb{R}$ for each $i=1,...,n$. These sets must be chosen such that the mapping \eqref{mapT} admits a fixed point $\theta_{i,\lambda}\in\mathcal{I}_{i,\lambda}$, and all $r_i+\theta_{i,\lambda}$ remain distinct. Based on the asymptotic estimates of $\Phi_{\lambda}$ and $\Psi_{\lambda}$ as $\lambda\to\infty$ (established in Section~\ref{sec-hs}), we shall set $\mathcal{I}_{i,\lambda}=:[-\lambda^{\frac{1-p}{4p}}, \lambda^{\frac{1-p}{4p}}]$ for each $i=1,...,n$. We will then show that, under assumptions (A1)--(A3), the mapping~$\boldsymbol{\mathsf{T}}_{i,\lambda}$ defined by \eqref{mapT}  is a well-defined contraction from $[-\lambda^{\frac{1-p}{4p}}, \lambda^{\frac{1-p}{4p}}]$ into itself for all sufficiently large $\lambda>0$.

\begin{proposition}\label{prop-t}
Assume \eqref{dhb}--\eqref{N-2t} and (A1)--(A2), and that the set $\boldsymbol{\mathrm{S}}_{\star}$ defined in \eqref{s0n} satisfies (A3). Then for each $r_i\in\boldsymbol{\mathrm{S}}_{\star}$, there exists $\lambda_{i,*}>0$ such that for $\lambda>\lambda_{i,*}$, $\boldsymbol{\mathsf{T}}_{i,\lambda}$ defined in \eqref{mapT} with $\mathcal{I}_{i,\lambda}=:[-\lambda^{\frac{1-p}{4p}}, \lambda^{\frac{1-p}{4p}}]$ admits a unique fixed point $\theta_{i,\lambda}$. 
\end{proposition}
The proof of Proposition~\ref{prop-t} is {\color{black} given} in Section~\ref{sec-propt}.

Proposition~\ref{prop-t} guarantees the existence of a unique fixed point~$\theta_{i,\lambda}\in[-\lambda^{\frac{1-p}{4p}}, \lambda^{\frac{1-p}{4p}}]$. This, in turn, demonstrate that for sufficiently large  $\lambda>0$, equation~\eqref{eq2u} with the boundary condition \eqref{rbd} has multiple solutions. Our first main result on the multiplicity and qualitative properties of solutions is stated as follows.

\begin{theorem}\label{thm0}
Under the assumptions \eqref{dhb}--\eqref{N-2t} and (A1)--(A3), {\color{black} equation~\eqref{eq2u} with the boundary condition~\eqref{rbd} admits multiple solutions associated with $\lambda_{i,*}$ and the fixed points $\theta_{i,\lambda}$ obtained in Proposition~\ref{prop-t}. More precisely, the following properties (i) and (ii) hold:}
\begin{itemize}
    \item[(i)]  If $\boldsymbol{\mathrm{S}}_{\star}=\{r_1,...,r_n\}$ is a finite set {\color{black} with the ordering~$r_1<r_2<\cdots<r_n$}, then there exists $\lambda_{*}\geq\max\{\lambda_{i,*}|\,i=1,...,n\}$ such that for $\lambda\geq\lambda_{*}$,  the equation has at least $n$ distinct solutions given explicitly by  
\begin{equation}\label{2.3u}  
 u_{i,\lambda}=(r_i+\theta_{i,\lambda})\Phi_{\lambda}+\Psi_{\lambda}\in\mathrm{C}^1(\overline{\Omega})\cap\mathrm{C}^2(\Omega),\,\,\text{for}\quad\.i=1,...,n,  
\end{equation}  
where $\Phi_{\lambda}$ and $\Psi_{\lambda}$ are the unique classical solutions of \eqref{eq-Phi} and \eqref{eq-Psi}, respectively. 
For each solution~$u_{i,\lambda}$, we have the following asymptotic behavior:
\begin{equation}\label{u-1as}
 |\boldsymbol{\mathsf{N}}[u_{i,\lambda}]-r_i| +\max_{\overline{\Omega}}  |u_{i,\lambda}-(r_i\Phi_{\lambda}+\Psi_{\lambda})|\xrightarrow{\lambda\to\infty}0.
\end{equation}
Moreover, if $b(x)>0$ on $\partial\Omega$, there holds $u_{1,\lambda}(x)<\cdots<u_{n,\lambda}(x)$ for $x\in\overline{\Omega}$.
\item[(ii)] If $\boldsymbol{\mathrm{S}}_{\star}$ has infinitely many distinct elements, then for any given $\widetilde{n}\in\mathbb{N}$, there exists $\lambda_{\widetilde{n}}>0$ such that for $\lambda\geq\lambda_{\widetilde{n}}$, the equation possesses at least $\widetilde{n}$ distinct solutions.
\end{itemize}
\end{theorem}

Theorem~\ref{thm0}(ii) states that  equation~\eqref{eq2u} with the boundary condition~\eqref{rbd} can admit an arbitrary finite number of distinct solutions. In this context, both the function~$\mathcal{N}$ and the parameter $\lambda$ play a crucial role in determining the number of solutions. This result represents a distinctive characteristic of this class of linear equations with nonlocal nonlinearity. The proof of Theorem~\ref{thm0} is {\color{black} given} in Section~\ref{thmc}.
\subsection{On the general case of multiple nonlocal nonlinearities}\label{sec2.3} Based on the approach established in Sections~\ref{sub2-1}--\ref{sub2-2}, the main purpose of this study is to generalize \eqref{eq-u} to a linear elliptic equation with multiple nonlocal terms, which is given by
\begin{equation}\label{eq-u2}
    -\nabla\cdot(\od(x)\nabla u) +\lambda f(x)u=   \eta \sum_{j=1}^kh_j(x)\boldsymbol{\mathsf{N}}_j[u]+h_0(x) \quad\text{in}\,\,\Omega,
\end{equation}
where $k\geq1$ and we make the following assumptions:
\begin{itemize}
    \item[\bf (A4).] For $j=0,1,...,k$, the functions $h_j\in\text{C}^{0}(\overline{\Omega};\mathbb{R})$ and $\frac{h_j}f\in\text{C}^{2}(\overline{\Omega};\mathbb{R})$. In addition, there exists an interior point $\boldsymbol{\mathrm{x}}_0\in\Omega$ where the set $\{h_1(\boldsymbol{\mathrm{x}}_0),...,h_k(\boldsymbol{\mathrm{x}}_0)\}$ is linearly independent.
  \item[\bf (A5).] Let $w_j\in\mathrm{L}^p(\Omega)$ for some $p>1$ be such that $\int_\Omega w_j(y)\dy=1$. Let $\mathcal{N}_j:\mathbb{R}\to\mathbb{R}$ be locally Lipschitz continuous. We define the nonlocal term $\boldsymbol{\mathsf{N}}_j[u]$ analogously to \eqref{N-2t} as follows:
\begin{equation}\label{N-2j}
\boldsymbol{\mathsf{N}}_j[u]= \begin{cases}
\displaystyle\int_\Omega w_j(y)\mathcal{N}_j(u(y))\dy & (\text{type~I});\vspace{3pt} \\
\displaystyle\mathcal{N}_j(\int_\Omega w_j(y)u(y)\dy) & (\text{type~II}).
\end{cases}
\end{equation}
Furthermore, we assume that the set of functions $\{w_1,...,w_k\}$ is linearly independent on $\overline{\Omega}$. Similarly, the set of functions $\{\mathcal{N}_1,...,\mathcal{N}_k\}$ is also linearly independent on $\mathbb{R}$. These assumptions guarantee that the nonlocal terms $\boldsymbol{\mathsf{N}}_j[u]$ cannot be combined or simplified.
\end{itemize}

To investigate the multiplicity of solutions to equation~\eqref{eq-u2} with the boundary condition~\eqref{rbd}, we will extend the methodology and argument established in Sections~\ref{sub2-1}--\ref{sub2-2}. Consider the following auxiliary equations:
\begin{equation}\label{eq-Phj}
    \begin{cases}
-\nabla\cdot(\od(x)\nabla \Phi_{j,\lambda,\eta}) + \lambda f(x)\Phi_{j,\lambda,\eta}=\eta h_j(x) \quad&\text{in}\,\,\Omega,\vspace{3pt}\\ 
\displaystyle\Phi_{j,\lambda,\eta}+b(x) \frac{\partial\Phi_{j,\lambda,\eta}}{\partial{\vec{\nu}}}  = 0\quad&\text{on}\,\,\partial\Omega,
    \end{cases}
\end{equation}
for $j=1,...,k$, and
\begin{equation}\label{eq-Ps0}
    \begin{cases}
-\nabla\cdot(\od(x)\nabla \Psi_{0,\lambda}) + \lambda f(x)\Psi_{0,\lambda}=h_0(x) \quad&\text{in}\,\,\Omega,\vspace{3pt}\\ 
\displaystyle\Psi_{0,\lambda}+b(x) \frac{\partial\Psi_{0,\lambda}}{\partial{\vec{\nu}}}  = g(x)\quad&\text{on}\,\,\partial\Omega.
    \end{cases}
\end{equation}
Since $f>0$, based on (A4) and \eqref{eq-Phj}, we can easily deduce that the unique solutions~$\Phi_{1,\lambda,\eta}$,...,$\Phi_{k,\lambda,\eta}$ are linearly independent. This, along with \eqref{eq-Ps0}, allows us to conclude that the representation 
\begin{equation}\label{rp-u}
u_{\lambda,\eta}(x)=\sum\limits_{j=1}^k\boldsymbol{\mathsf{N}}_j[u_{\lambda,\eta}]\Phi_{j,\lambda,\eta}(x)+\Psi_{0,\lambda}(x)  
\end{equation}
 solves equation~\eqref{eq-u2} with the boundary condition~\eqref{rbd}. Hence, the original multiplicity problem reduces to finding multiple solutions~$u$ to the implicit integral equation~\eqref{rp-u} with multiple nonlocal terms~$\boldsymbol{\mathsf{N}}_j[u_{\lambda,\eta}]$ for $j=1,...,k$.

In what follows, we shall investigate the existence of multiple solutions to \eqref{rp-u}. We define the vector-valued function
\begin{equation}\label{HN}
\overrightarrow{\boldsymbol{\mathrm{H}}}(x):=\big\langle \frac{h_1(x)}{f(x)},...,\frac{h_k(x)}{f(x)} \big\rangle,\quad~x\in\overline{\Omega},
\end{equation}
and denote the inner product of vectors $\vec{\boldsymbol{a}},\,\vec{\boldsymbol{b}}\in\mathbb{R}^k$ as $\vec{\boldsymbol{a}}\boldsymbol{\cdot}\vec{\boldsymbol{b}}$, with the associated norm given by $|\vec{\boldsymbol{a}}|=\sqrt{\vec{\boldsymbol{a}}\boldsymbol{\cdot}\vec{\boldsymbol{a}}}$. Under an asymptotic viewpoint and using fixed point arguments, we will show that the original multiplicity problem is now linked to the number of distinct solutions~~$\vec{\boldsymbol{\mathrm{r}}}\in\mathbb{R}^k$ of the following  equation:
\begin{equation}\label{rNH}
\vec{\boldsymbol{\mathrm{r}}}=\big\langle\boldsymbol{\mathsf{N}}_1\big[\vec{\boldsymbol{\mathrm{r}}}\boldsymbol{\cdot}\overrightarrow{\boldsymbol{\mathrm{H}}}\big],...,\boldsymbol{\mathsf{N}}_k\big[\vec{\boldsymbol{\mathrm{r}}}\boldsymbol{\cdot}\overrightarrow{\boldsymbol{\mathrm{H}}}\big] \big\rangle.
\end{equation}
By \eqref{N-2j}, each $\boldsymbol{\mathsf{N}}_j\big[\vec{\boldsymbol{\mathrm{r}}}\boldsymbol{\cdot}\overrightarrow{\boldsymbol{\mathrm{H}}}\big]$  is a real number independent of $x$. As a result, \eqref{rNH} constitutes a system of algebraic equations for the unknown vector $\vec{\boldsymbol{\mathrm{r}}}$.

Firstly, we consider the set $ \boldsymbol{\mathrm{S}}_k$ consisting of solutions $\vec{\boldsymbol{\mathrm{r}}}$ to \eqref{rNH} with the following properties:

\begin{equation}\label{sk}
\boldsymbol{\mathrm{S}}_k= \left\{\vec{\boldsymbol{\mathrm{r}}}\in\mathbb{R}^k\left|\,
\begin{aligned}
&\vec{\boldsymbol{\mathrm{r}}}\,\,\text{satisfies}~(\ref{rNH}),\,\,\text{and~there~exists}~\delta_{\vec{\boldsymbol{\mathrm{r}}}}>0\,\,\text{such~that}\\
&|\mathcal{N}_j({s}_1)-\mathcal{N}_j({s}_2)|\leq\mathcal{L}_{\vec{\boldsymbol{\mathrm{r}}}}|{s}_1-{s}_2|,\,\,\text{for\,all}\, j=1,...,k,\\
&\text{whenever}\,\,||{s}_i-\vec{\boldsymbol{\mathrm{r}}}\boldsymbol{\cdot}\overrightarrow{\boldsymbol{\mathrm{H}}}||_{\mathrm{L}^{\infty}(\Omega)}<\delta_{\vec{\boldsymbol{\mathrm{r}}}},\text{ for }i=1,2.
\end{aligned}\right.\right\},
\end{equation}
where $\delta_{\vec{\boldsymbol{\mathrm{r}}}}$ depends on $\vec{\boldsymbol{\mathrm{r}}}$, and each $\mathcal{L}_{\vec{\boldsymbol{\mathrm{r}}}}$ satisfies
\begin{equation}\label{Lr}
0<\mathcal{L}_{\vec{\boldsymbol{\mathrm{r}}}}|\Omega|^{1-\frac1p}\max_{1\leq j\leq red}{k}||w_j||_{\mathrm{L}^{p}(\Omega)} <\frac1{k||\overrightarrow{\boldsymbol{\mathrm{H}}}||_{\mathrm{L}^{\infty}(\Omega)}}.
\end{equation}
Here, the $\text{L}^\infty$-norm of a $k$-dimensional vector-valued function is defined as the maximum of the $\text{L}^\infty$-norms of its $k$ components.

In our subsequent analysis of the existence of multiple solutions to \eqref{rp-u}, we assume the following regarding \eqref{sk}:
\begin{itemize}
    \item[\bf(A6).] The set  $\boldsymbol{\mathrm{S}}_k$ has $n$ distinct elements, where $n\geq2$ and $n$ can be infinite. 
\end{itemize}
In Appendix (Section~\ref{BPX}), we provide two examples to verify the validity of \eqref{Lr} and assumption~(A6).

In Section~\ref{sub2-2}, we have established the existence of multiple solutions for equation~\eqref{eq2u} with boundary condition~\eqref{rbd} for sufficiently large $\lambda$. This was achieved by combining fixed-point arguments and asymptotic analysis under assumption~(A3) for the set $\boldsymbol{\mathrm{S}}_{\star}$ defined in \eqref{s0n}. Based on these findings, we now proceed directly to define the following mapping on a suitable complete metric space (to be specified in the proof of Theorem~\ref{thm2}): 
\begin{equation}\label{nT}
\boldsymbol{\Gamma}_{\vec{\boldsymbol{\mathrm{r}}},\lambda,\eta}(\vec{\boldsymbol{\vartheta}}) =\big\langle\boldsymbol{\mathsf{N}}_1\big[(\vec{\boldsymbol{\mathrm{r}}}+\vec{\boldsymbol{\vartheta}})\boldsymbol{\cdot}\overrightarrow{\boldsymbol{\Phi}}_{\lambda,\eta}+\Psi_{0,\lambda}\big],...,\boldsymbol{\mathsf{N}}_k\big[(\vec{\boldsymbol{\mathrm{r}}}+\vec{\boldsymbol{\vartheta}})\boldsymbol{\cdot}\overrightarrow{\boldsymbol{\Phi}}_{\lambda,\eta}+\Psi_{0,\lambda}\big] \big\rangle-\vec{\boldsymbol{\mathrm{r}}}
\end{equation}
where 
\begin{equation}\label{vP}
\overrightarrow{\boldsymbol{\Phi}}_{\lambda,\eta}:=\big\langle\Phi_{1,\lambda,\eta},...,\Phi_{k,\lambda,\eta}\big\rangle.   
\end{equation}
We will prove that the mapping $\boldsymbol{\Gamma}_{\vec{\boldsymbol{\mathrm{r}}},\lambda,\eta}$ is a contraction in a suitable complete metric space for sufficiently large $\lambda$ and $\eta$, thus admitting a fixed point~$\vec{\boldsymbol{\vartheta}}_{\vec{\boldsymbol{\mathrm{r}}},\lambda,\eta}$ (cf. Proposition~\ref{prop-3} in Section~\ref{sec-thm2}). It follows from \eqref{nT} that
\begin{equation}\label{nPh}
\vec{\boldsymbol{\mathrm{r}}}+\vec{\boldsymbol{\vartheta}}_{\vec{\boldsymbol{\mathrm{r}}},\lambda,\eta}=\big\langle\boldsymbol{\mathsf{N}}_1\big[(\vec{\boldsymbol{\mathrm{r}}}+\vec{\boldsymbol{\vartheta}}_{\vec{\boldsymbol{\mathrm{r}}},\lambda,\eta})\boldsymbol{\cdot}\overrightarrow{\boldsymbol{\Phi}}_{\lambda,\eta}+\Psi_{0,\lambda}\big],...,\boldsymbol{\mathsf{N}}_k\big[(\vec{\boldsymbol{\mathrm{r}}}+\vec{\boldsymbol{\vartheta}}_{\vec{\boldsymbol{\mathrm{r}}},\lambda,\eta})\boldsymbol{\cdot}\overrightarrow{\boldsymbol{\Phi}}_{\lambda,\eta}+\Psi_{0,\lambda}\big] \big\rangle.   
\end{equation}
 Now we define
\begin{equation}\label{rtu}
u_{\vec{\boldsymbol{\mathrm{r}}},\lambda,\eta}=(\vec{\boldsymbol{\mathrm{r}}}+\vec{\boldsymbol{\vartheta}}_{\vec{\boldsymbol{\mathrm{r}}},\lambda,\eta})\boldsymbol{\cdot}\overrightarrow{\boldsymbol{\Phi}}_{\lambda,\eta}+\Psi_{0,\lambda}.  
\end{equation}
Combining \eqref{vP}--\eqref{nPh} with \eqref{rtu} yields
\begin{equation*}
u_{\vec{\boldsymbol{\mathrm{r}}},\lambda,\eta}=\sum\limits_{j=1}^k\boldsymbol{\mathsf{N}}_j\big[(\vec{\boldsymbol{\mathrm{r}}}+\vec{\boldsymbol{\vartheta}}_{\vec{\boldsymbol{\mathrm{r}}},\lambda,\eta})\boldsymbol{\cdot}\overrightarrow{\boldsymbol{\Phi}}_{\lambda,\eta}+\Psi_{0,\lambda}\big]\Phi_{j,\lambda,\eta}+\Psi_{0,\lambda}=\sum\limits_{j=1}^k\boldsymbol{\mathsf{N}}_j[u_{\vec{\boldsymbol{\mathrm{r}}},\lambda,\eta}]\Phi_{j,\lambda,\eta}+\Psi_{0,\lambda}, 
\end{equation*}
 which is precisely the form of \eqref{rp-u}. As a consequence, $u_{\vec{\boldsymbol{\mathrm{r}}},\lambda,\eta}$ defined in \eqref{rtu} forms a solution to equation~\eqref{eq-u2} with the boundary condition~\eqref{rbd}.

We will show that if $\vec{\boldsymbol{\mathrm{r}}}\in\boldsymbol{\mathrm{S}}_k$, then for sufficiently large $\lambda>0$, the mapping $\vec{\boldsymbol{\mathrm{r}}}\mapsto u_{\vec{\boldsymbol{\mathrm{r}}},\lambda,\eta}$ is injective. This injectivity, along with assumption~(A6), implies that the equation has at least $n$ distinct solutions. The corresponding result is stated as follows.
\begin{theorem}\label{thm2}
Assume \eqref{dhb} and (A4)--(A6), and that the set $\boldsymbol{\mathrm{S}}_k$ defined in \eqref{sk} satisfies (A6). Let $L>0$ and $\tau<0$ be arbitrary given constants. {\color{black} Then equation~\eqref{eq-u2} with the boundary condition~\eqref{rbd} admits multiple solutions of the form \eqref{rtu}. More precisely, the following properties (i) and (ii) hold:}
\begin{itemize}
    \item[(i)]  If $\boldsymbol{\mathrm{S}}_k=\{\vec{\boldsymbol{\mathrm{r}}}_1,...,\vec{\boldsymbol{\mathrm{r}}}_n\}\subset\mathbb{R}^k$ is a finite set with exactly $n$ distinct elements, then there exists a positive constant $\lambda^{*}$ depending on $L$, $\tau$ and all $\vec{\boldsymbol{\mathrm{r}}}_j$'s such that for $\lambda\geq\lambda^{*}$ and 
    \begin{equation}\label{le-*i}
\left|{\eta}-{\lambda}\right|\leq L\lambda^{\tau+\frac{1+3p}{4p}},
    \end{equation}
 the equation has at least $n$ distinct solutions~$u_{\vec{\boldsymbol{\mathrm{r}}}_j,\lambda,\eta}$ given explicitly by  \eqref{rtu} with $\vec{\boldsymbol{\mathrm{r}}}=\vec{\boldsymbol{\mathrm{r}}}_j$, $j=1,...,n$, where $|\vec{\boldsymbol{\vartheta}}_{\vec{\boldsymbol{\mathrm{r}}}_j,\lambda,\eta}|\leq{k}^{-1}\lambda^{\frac{1-p}{4p}}$. In particular, $u_{\vec{\boldsymbol{\mathrm{r}}}_j,\lambda,\eta}$ satisfies
\begin{equation}\label{u-1sa}
\left|\big\langle\boldsymbol{\mathsf{N}}_1\big[u_{\vec{\boldsymbol{\mathrm{r}}}_j,\lambda,\eta}\big],...,\boldsymbol{\mathsf{N}}_k\big[u_{\vec{\boldsymbol{\mathrm{r}}}_j,\lambda,\eta}\big] \big\rangle-\vec{\boldsymbol{\mathrm{r}}}_j\right| +\max_{\overline{\Omega}}  \left|u_{\vec{\boldsymbol{\mathrm{r}}}_j,\lambda,\eta}-\left(\vec{\boldsymbol{\mathrm{r}}}_j\boldsymbol{\cdot}\overrightarrow{\boldsymbol{\Phi}}_{\lambda,\eta}+\Psi_{0,\lambda}\right)\right|\xrightarrow{\lambda\to\infty}0.
\end{equation}
\item[(ii)] If $\boldsymbol{\mathrm{S}}_k$ has infinitely many distinct elements, then for any given $\widetilde{n}\in\mathbb{N}$, there exists a positive constant $\lambda_{\widetilde{n}}^*$ such that for $\lambda\geq\lambda_{\widetilde{n}}^*$ and $\eta$ satisfies \eqref{le-*i}, the equation possesses at least $\widetilde{n}$ distinct solutions.
\end{itemize}
\end{theorem}

The proof of Theorem~\ref{thm2} is {\color{black} given} in Section~\ref{sec-thm2}.
\section{\bf Proof of Proposition~\ref{prop1}}\label{APX}
In this section, we state the proof of Proposition~\ref{prop1}. 
For $\Theta\in\mathbb{R}$, we consider the equation 
\begin{equation}\label{veq}
    -\nabla\cdot(\od(x)\nabla v) +\lambda f(x)v=\eta \Theta h(x) \quad\text{in}\,\,\Omega,
\end{equation}
with the boundary condition
\begin{equation}\label{vbd}
 v+b(x) \frac{\partial v}{\partial{\vec{\nu}}}  = g(x)\quad\text{on}\,\,\partial\Omega.
\end{equation}
Note that $\Theta$ corresponds to the nonlocal term in \eqref{eq-u}.
For each $\lambda,\eta>0$, by \eqref{dhb} and the standard elliptic regularity theorem~\cite{GT1983}, equation~\eqref{veq}--\eqref{vbd} admits a unique solution $v=v_{\lambda,\eta,\Theta}\in\text{C}^1(\overline{\Omega})\cap\text{C}^2(\Omega)$. This allows us to define a mapping $T_{\lambda,\eta}:\mathbb{R}\to\mathbb{R}$ as
\begin{equation}\label{mto-v}
    T_{\lambda,\eta}(\Theta)=\boldsymbol{\mathsf{N}}[v_{\lambda,\eta,\Theta}].
\end{equation}
 For $\Theta_1,\,\Theta_2\in\mathbb{R}$, we shall claim:
 \begin{itemize}
     \item[\bf(i)] if $b(x)\geq0$ on $\partial\Omega$, then 
\begin{equation}\label{app-t1}
    |T_{\lambda,\eta}(\Theta_1)-T_{\lambda,\eta}(\Theta_2)|\leq\frac{\eta\mathcal{L}\Omega|^{1-\frac1p}||w||_{\text{L}^{p}(\Omega)}}{\lambda}\max_{\overline{\Omega}}\frac{|h|}{f}|\Theta_1-\Theta_2|.
\end{equation}
\item[\bf(ii)] if $b(x)\equiv0$ on $\partial\Omega$, then 
\begin{equation}\label{app-t3}
\begin{aligned}
   & |T_{\lambda,\eta}(\Theta_1)-T_{\lambda,\eta}(\Theta_2)|\\
\leq&\,\frac{\eta\mathcal{L}|\Omega|^{1-\frac1p}||w||_{\text{L}^{p}(\Omega)}\max_{\overline{\Omega}}|h|}{\left(\lambda\min_{\overline{\Omega}}f\right)^{\max\{\frac12,\frac1p\}}\left(\mathfrak{D}\Lambda_{\Omega}+\lambda\min_{\overline{\Omega}}f\right)^{1-\max\{\frac12,\frac1p\}}}|\Theta_1-\Theta_2|. 
\end{aligned}
\end{equation}
 \end{itemize}
\begin{proof}[\bf Proof of \eqref{app-t1}] It suffices to consider $\Theta_1\neq\Theta_2$. Let $v_i=v_{\lambda,\eta,\Theta_i}$ be the unique solution of \eqref{veq}--\eqref{vbd} corresponding to $\Theta=\Theta_i$, $i=1,2$. Then we have $\max_{\overline{\Omega}}|v_1-v_2|>0$.
Using \eqref{dhb}, a simple calculation yields:
\begin{equation}\label{app-t2}
    \begin{aligned}
   &\lambda f(x)(v_1-v_2)^2-\eta(\Theta_1-\Theta_2)h(x)(v_1-v_2)\\
   =&\, (v_1-v_2)   \nabla\cdot(\od(x)\nabla(v_1-v_2))\\ \leq&\,\frac12\nabla\cdot\left(\od(x)\nabla(v_1-v_2)^2\right)-\mathfrak{D}|\nabla(v_1-v_2)|^2,\quad\forall x\in\Omega,
    \end{aligned}
\end{equation}
and
\begin{equation}\label{vbd-2}
 2(v_1-v_2)^2+b(x) \frac{\partial}{\partial{\vec{\nu}}}(v_1-v_2)^2  =0\quad\text{on}\,\,\partial\Omega.
\end{equation}  
Since $b(x)\geq0$ and $\max_{\overline{\Omega}}|v_1-v_2|>0$, \eqref{vbd-2} implies that $(v_1-v_2)^2$ attains its maximum value at an interior point $x_M\in\Omega$. By evaluating \eqref{app-t2} at $x=x_M$, we find
$\lambda f(x_M)(v_1(x_M)-v_2(x_M))^2-\eta(\Theta_1-\Theta_2)h(x_M)(v_1(x_M)-v_2(x_M))\leq0$. As a consequence,
\begin{equation}\label{v1v2}
\max_{\overline{\Omega}}|v_1-v_2|=   |v_1(x_M)-v_2(x_M)|\leq\frac{\eta|h(x_M)|}{\lambda f(x_M)} |\Theta_1-\Theta_2|.
\end{equation}

Note that $\mathcal{N}:\mathbb{R}\to\mathbb{R}$ satisfies $|\mathcal{N}(s_1)-\mathcal{N}(s_2)|\leq\mathcal{L}|s_1-s_2|$, $\forall~s_1,s_2\in\mathbb{R}$. Combining \eqref{N-2t}, \eqref{mto-v} and   \eqref{v1v2}, one may check that
\begin{align*}
    |T_{\lambda,\eta}(\Theta_1)-T_{\lambda,\eta}(\Theta_2)|=&\,|\boldsymbol{\mathsf{N}}[v_1]-\boldsymbol{\mathsf{N}}[v_2]|\\
\leq&\,\mathcal{L}\int_{\Omega}|w(y)||v_1(y)-v_2(y)|\,\text{d}y\\
\leq&\,\frac{\eta\mathcal{L}}{\lambda} \max_{\overline{\Omega}}\frac{|h|}{f}||w||_{\text{L}^p(\Omega)}|\Omega|^{1-\frac1p} |\Theta_1-\Theta_2|.  
\end{align*}
Here, we applied the H\"{o}lder's inequality to the second line and used \eqref{v1v2} for $||v_1-v_2||_{\text{L}^q(\Omega)}$ with $q\in[1,\infty)$ satisfying  $\frac1p+\frac1q=1$ to arrive at the last estimate. This completes the proof of~\eqref{app-t1}. 
\end{proof}
\begin{proof}[\bf Proof of \eqref{app-t3}] Assume $b(x)\equiv0$ on $\partial\Omega$. Then by \eqref{vbd-2} we have $(v_1-v_2)^2=0$ on $\partial\Omega$. By \eqref{dhb}, \eqref{app-t2} and \eqref{v1v2}, we have
\begin{equation}\label{iun-1}
\begin{aligned}
  &\left(\mathfrak{D}\Lambda_{\Omega}+\lambda\min_{\overline{\Omega}}f\right)\int_{\Omega}(v_1-v_2)^2\dx\\
\leq&\,\int_{\Omega}\left(\mathfrak{D}|\nabla(v_1-v_2)|^2+\lambda\min_{\overline{\Omega}}f(v_1-v_2)^2\right)\dx\\
  \leq&\, \frac12\int_{\Omega}\nabla\cdot\left(\od(x)\nabla(v_1-v_2)^2\right)  \dx  +\eta|\Theta_1-\Theta_2|\max_{\overline{\Omega}}|h|\int_{\Omega}|v_1-v_2|  \dx \\
\leq&\,\frac{\eta^2|\Omega|\max_{\overline{\Omega}}h^2}{\lambda\min_{\overline{\Omega}} f} |\Theta_1-\Theta_2|^2.
\end{aligned}
\end{equation}
Here, the second line of \eqref{iun-1} follows from the Poincar\'{e} inequality $\Lambda_{\Omega}||v_1-v_2||_{\text{L}^2(\Omega)}^2\leq||\nabla(v_1-v_2)||_{\text{L}^2(\Omega)}^2$, while the last line is obtained due to  $\int_{\Omega}\nabla\cdot\left(\od(x)\nabla(v_1-v_2)^2\right)  \dx=0$.

To estimate $\int_{\Omega}|w(y)||v_1(y)-v_2(y)|\,\dy$, we consider two cases $p>2$ and $1<p\leq 2$. When $p>2$, a direct application of H\"{o}lder's inequality yields 
\begin{equation}\label{iun-2}
    \begin{aligned}
\int_{\Omega}|w(y)||v_1(y)-v_2(y)|\dy    \leq&\,|\Omega|^{\frac{p-2}{2p}}||w||_{\text{L}^p(\Omega)} ||v_1-v_2||_{\text{L}^2(\Omega)}\\
(\text{by\,\,(\ref{iun-1})})\,\,\leq&\,\frac{\eta|\Omega|^{1-\frac1{p}}||w||_{\text{L}^p(\Omega)}\max_{\overline{\Omega}}|h|}{\sqrt{\lambda\min_{\overline{\Omega}} f\left(\mathfrak{D}\Lambda_{\Omega}+\lambda\min_{\overline{\Omega}}f\right)}} |\Theta_1-\Theta_2|.
    \end{aligned}
\end{equation}
On the other hand, when $1<p\leq2$, we have
\begin{equation}\label{iun-3}
    \begin{aligned}
\int_{\Omega}|w(y)||v_1(y)-v_2(y)|\dy    \leq&\,||w||_{\text{L}^p(\Omega)} ||v_1-v_2||_{\text{L}^{\frac{p}{p-1}}(\Omega)}\\
\leq&\,||w||_{\text{L}^p(\Omega)}||v_1-v_2||_{\text{L}^{2}(\Omega)}^{2(1-\frac1p)} ||v_1-v_2||_{\text{L}^{\infty}(\Omega)}^{\frac2p-1} \\
(\text{by\,\,(\ref{v1v2})\,\,and\,\,(\ref{iun-1})})\,\,\leq&\,\frac{\eta|\Omega|^{1-\frac1{p}}||w||_{\text{L}^p(\Omega)}\max_{\overline{\Omega}}|h|}{\left(\lambda\min_{\overline{\Omega}}f\right)^{\frac1p}\left(\mathfrak{D}\Lambda_{\Omega}+\lambda\min_{\overline{\Omega}}f\right)^{1-\frac1p}} |\Theta_1-\Theta_2|.
    \end{aligned}
\end{equation}
By combining $|T_{\lambda,\eta}(\Theta_1)-T_{\lambda,\eta}(\Theta_2)|=|\boldsymbol{\mathsf{N}}[v_1]-\boldsymbol{\mathsf{N}}[v_2]|\leq\mathcal{L}\int_{\Omega}|w(y)||v_1(y)-v_2(y)|\,\text{d}y$ with \eqref{iun-2} and \eqref{iun-3}, we obtain \eqref{app-t3} and complete the proof.
\end{proof}

For the case $b(x)\geq0$ on $\partial\Omega$, we see from \eqref{le=} and \eqref{app-t1} that $T_{\lambda,\eta}:\mathbb{R}\to\mathbb{R}$  is a contraction mapping, which ensures the existence of a unique fixed point. In the special case where $b(x)\equiv0$ on $\partial\Omega$,  it follows from \eqref{=le} and \eqref{app-t3} that the map also admits a unique fixed point. We use the same symbol $\Theta_{\lambda,\eta}$ for the fixed point in both cases for simplicity.  This fixed point satisfies the relation
\begin{equation*}
  \Theta_{\lambda,\eta}=T_{\lambda,\eta}(\Theta_{\lambda,\eta})=\boldsymbol{\mathsf{N}}[v_{\lambda,\eta,\Theta_{\lambda,\eta}}].  
\end{equation*}
  Using this result in combination with \eqref{veq}--\eqref{vbd}, we show that $u=v_{\lambda,\eta,\Theta_{\lambda,\eta}}$ is a solution of~\eqref{eq-u}--\eqref{rbd}.

It remains to prove the uniqueness of the solution to \eqref{eq-u}--\eqref{rbd} under the conditions on $\lambda$ and $\eta$ given by \eqref{le=}. Suppose, for the sake of contradiction, that \eqref{eq-u}--\eqref{rbd} has two distinct solutions $u_1^*$ and $u_2^*$. We define
\begin{equation*}
\Theta_i^*:=\boldsymbol{\mathsf{N}}[u_i^*]\quad\text{and}\quad\,v_i^*:=v_{\lambda,\eta,\Theta_i},\,\,i=1,2.   
\end{equation*}
Then, by the uniqueness of \eqref{veq}--\eqref{vbd}, the solution corresponding to
$\Theta=\Theta_i^*$ is unique, so~$v_i^*$ must be equal to $u_i^*$. Along with \eqref{mto-v}, we have
\begin{equation*}
u_i^*=v_i^*\quad\text{and}\quad\Theta_i^*=\boldsymbol{\mathsf{N}}[v_i^*]=T_{\lambda,\eta}(\Theta_i^*),\quad~i=1,2.
\end{equation*}
This shows that both $\Theta_1^*$ and $\Theta_2^*$ are fixed points of $T_{\lambda,\eta}$ defined in \eqref{mto-v}. Since $T_{\lambda,\eta}$ is a contraction mapping with a unique fixed point over $\mathbb{R}$, we conclude that $\Theta_1^*=\Theta_2^*$. This, in turn, implies $v_1^*=v_{\lambda,\eta,\Theta_1}=v_{\lambda,\eta,\Theta_2}=v_2^*$. As a consequence, we have $u_1^*=v_1^*=v_2^*=u_2^*$ which is a contradiction to our initial assumption that the solutions were distinct. Therefore, we prove the uniqueness of the solution to \eqref{eq-u}--\eqref{rbd} and complete the proof of Proposition~\ref{prop1}.
\section{\bf Proof of Proposition~\ref{prop-t} and Theorem~\ref{thm0}}\label{sec-thm0}
\subsection{Asymptotic estimates of $\Phi_{\lambda}$ and $\Psi_{\lambda}$}\label{sec-hs}
Multiplying the equation in \eqref{eq-Phi} by $\Phi_{\lambda}-1$ and using \eqref{dhb}, one may check that
\begin{equation}\label{ch-8}
    \begin{aligned}
 \lambda (\Phi_{\lambda}-1)^2=&\,\frac1{f(x)}(\Phi_{\lambda}-1)\nabla\cdot(\od(x)\nabla \Phi_{\lambda}) \\
\leq&\,\frac1{f(x)}\left(\frac12\nabla\cdot\left(\od(x)\nabla(\Phi_{\lambda}-1)^2\right)-\mathfrak{D}|\nabla\Phi_{\lambda}|^2\right)\\
\leq&\,\left(2\min_{\overline{\Omega}}f\right)^{-1}\nabla\cdot\left(\od(x)\nabla(\Phi_{\lambda}-1)^2\right)\quad\text{in}\,\,\Omega.
    \end{aligned}
\end{equation}

We further establish a refined upper bound depending on $\lambda>0$ for $\Phi_{\lambda}$. According to \eqref{Phs} and \eqref{ch-8}, we consider an auxiliary function
\begin{equation}\label{w-V}
V(x)=\exp\left({-\mathfrak{M}_0\sqrt{\lambda}\mathsf{d}(x)}\right),\quad\,x\in\overline{\Omega},
\end{equation}
which is chosen to satisfy $V\geq(\Phi_{\lambda}-1)^2$ on $\partial\Omega$,
where $\mathsf{d}(x):=\text{dist}(x,\partial\Omega)$ and $\mathfrak{M}_0$ (independent of $\lambda$) is a positive constant to be determined later. We will derive a differential inequality for $V$ that matches the form of \eqref{ch-8}. To do this, we first introduce a key concept for a rigorous argument, noting that the boundary $\partial\Omega$ is smooth. This smoothness ensures that for a sufficiently small $\delta>0$ (independent of $\lambda$), the subdomain $\Omega_{\delta}=\{x\in\Omega\,|\,\mathsf{d}(x)<\delta\}$ is free from any focal points, and the distance function $\mathsf{d}\in\text{C}^2(\overline{\Omega_{\delta}})$ (cf. \cite{CR2015}). Using \eqref{w-V}, we can establish estimates for $\nabla\cdot\left(\nabla(V-(\Phi_{\lambda}-1)^2)\right)$ within $\Omega_{\delta}$ and then choose a constant $\mathfrak{M}_0$ such that $(\Phi_{\lambda}-1)^2(x)\leq {V}(x)$ for all $x\in\overline{\Omega_{\delta}}$, provided $\lambda$ is sufficiently large. Subsequently, a standard application of the maximum principle  (since $\min_{\overline{\Omega}}f>0$) extends this result to the entire domain, yielding $(\Phi_{\lambda}-1)^2(x)\leq{V}(x)$ for all $x\in\overline{\Omega}$. For simplicity, we will present a simplified argument by considering the case where $\overline{\Omega}$ has no focal points and $\mathsf{d}\in\text{C}^2(\overline{\Omega})$. Then, we have $|\nabla\mathsf{d}(x)|=1$ for  $x\in\Omega$, which streamlines our analysis.

We now use the preceding arguments to derive detailed estimates. A direct computation, based on the fact from \eqref{w-V} that $\nabla{V}=-\sqrt{\lambda}\mathfrak{M}_0{V}\nabla\mathsf{d}$ and $|\nabla\mathsf{d}|=1$, yields
\begin{equation}\label{mkv}
  \begin{aligned}
&\nabla\cdot\left(\od(x)\nabla{V}(x)\right)\\ 
=&\,\left[\lambda{\mathfrak{M}_0^2\nabla\mathsf{d}(x)\cdot\od(x)\cdot(\nabla\mathsf{d}(x))^T}-\sqrt{\lambda}\mathfrak{M}_0\nabla\cdot(\od(x)\nabla\mathsf{d}(x))\right]{V}(x)\\
\leq&\,\mathfrak{M}_1\left(\lambda{\mathfrak{M}_0^2}+\sqrt{\lambda}\mathfrak{M}_0\right){V}(x)\quad\text{in}\,\,\Omega,
\end{aligned}  
\end{equation}
where $(\nabla\mathsf{d})^T$ denotes the transpose of $\nabla\mathsf{d}$ and
\begin{equation}\label{m11}
\mathfrak{M}_1=\max\left\{\max_{1\leq i,j\leq {\color{black}m}}\max\limits_{\overline{\Omega}}|D_{ij}|,\max\limits_{\overline{\Omega}}\left|\nabla\cdot(\od\nabla\mathsf{d})\right|\right\}. 
\end{equation}
From \eqref{ch-8} and \eqref{mkv}--\eqref{m11}, we need to ensure a positive constant~$\mathfrak{M}_0$ such that  $\mathfrak{M}_1\left(\lambda{\mathfrak{M}_0^2}+\sqrt{\lambda}\mathfrak{M}_0\right)\leq2\lambda\min_{\overline{\Omega}}f$ holds for sufficiently large $\lambda$. To fulfill this requirement, we can choose
\begin{equation}\label{ml0}
\mathfrak{M}_0=\sqrt{\frac{1}{\mathfrak{M}_1}\min_{\overline{\Omega}}f}\,\,\text{and}\,\,   \lambda\geq\frac{1}{\mathfrak{M}_0^2}.
\end{equation}
Then, a simple calculation shows the required estimate
\begin{equation*}
   2\lambda\min_{\overline{\Omega}}f-\mathfrak{M}_1\left(\lambda{\mathfrak{M}_0^2}+\sqrt{\lambda}\mathfrak{M}_0\right)=\sqrt{\lambda}\left(\sqrt{\lambda}\min_{\overline{\Omega}}f-\mathfrak{M}_0\mathfrak{M}_1\right)  
    \geq0. 
\end{equation*}
 Under \eqref{ml0}, we thus arrive at 
\begin{align*}
0&\,\leq-\nabla\cdot\left(\od(x)\nabla\left({V}-(\Phi_{\lambda}-1)^2\right)\right)\\
&\qquad\quad+  \mathfrak{M}_1\left(\lambda{\mathfrak{M}_0^2}+\sqrt{\lambda}\mathfrak{M}_0\right)V-2\lambda\min_{\overline{\Omega}}f(\Phi_{\lambda}-1)^2\\
&\,\leq2\lambda\min_{\overline{\Omega}}f\left({V}-(\Phi_{\lambda}-1)^2\right)\quad\text{in}\,\,\Omega.    
\end{align*}
Along with $V\geq(\Phi_{\lambda}-1)^2$ on $\partial\Omega$, we conclude that for $\lambda\geq\frac{1}{\mathfrak{M}_0^2}$, the following estimate holds:
 \begin{equation}\label{ch-9}
|\Phi_{\lambda}(x)-1|\leq\sqrt{V(x)}=\exp\left({-\frac{\mathsf{d}(x)}2\mathfrak{M}_0\sqrt{\lambda}}\right),\quad\forall\,x\in\overline{\Omega}.   
 \end{equation} 
 Using \eqref{Phs}, we can apply the same argument as for \eqref{ch-8}--\eqref{ch-9} to the equation \eqref{eq-Psi} to obtain the estimate
 \begin{equation}\label{ch-10}
|\Psi_{\lambda}(x)|\leq\max_{\partial\Omega}|g|\exp\left({-\frac{\mathsf{d}(x)}2\mathfrak{M}_0\sqrt{\lambda}}\right),\quad\forall\,x\in\overline{\Omega},   
 \end{equation}
as $\lambda\geq\frac{1}{\mathfrak{M}_0^2}$.
\subsection{\bf Proof of Proposition~\ref{prop-t}}\label{sec-propt} Let $r_i\in\boldsymbol{\mathrm{S}}_{\star}$ be fixed. Firstly, we need a lemma:
\begin{lemma}\label{lem1}
    Under the same hypotheses as in Proposition~\ref{prop-t}, there exists $\widetilde{\lambda}_{i,*}\geq\frac1{\mathfrak{M}_0^2}$ such that for each $\lambda\geq\widetilde{\lambda}_{i,*}$,
    \begin{equation}\label{T*1}
   \boldsymbol{\mathsf{T}}_{i,\lambda}([-\lambda^{\frac{1-p}{4p}}, \lambda^{\frac{1-p}{4p}}])\subseteq[-\lambda^{\frac{1-p}{4p}}, \lambda^{\frac{1-p}{4p}}].     
    \end{equation}
\end{lemma}
\begin{proof}
 For $r_i\in\boldsymbol{\mathrm{S}}_{\star}$ and $|\theta|\leq\lambda^{\frac{1-p}{4p}}$,  by \eqref{N-2t}, (A2) and \eqref{s0n}, we have
 \begin{equation}\label{ch-12}
 \begin{aligned}
 &\, \boldsymbol{\mathsf{N}}[(r_i+\theta)\Phi_{\lambda}+\Psi_{\lambda}]-r_i=
\boldsymbol{\mathsf{N}}[(r_i+\theta)\Phi_{\lambda}+\Psi_{\lambda}]-\mathcal{N}(r_i) \\= &\,\begin{cases}
 \displaystyle\int_\Omega w(y)\left(\mathcal{N}((r_i+\theta)\Phi_{\lambda}(y)+\Psi_{\lambda}(y))-\mathcal{N}(r_i)\right)\dy&(\text{type~I});\vspace{3pt} \\
 \displaystyle\mathcal{N}(\int_\Omega w(y)((r_i+\theta)\Phi_{\lambda}(y)+\Psi_{\lambda}(y))\dy)-\mathcal{N}(r_i)\qquad&(\text{type~II}).
    \end{cases}
 \end{aligned}
 \end{equation}
 According to \eqref{mapT}, we must individually estimate the term $| \boldsymbol{\mathsf{N}}[(r_i+\theta)\Phi_{\lambda}+\Psi_{\lambda}]-\mathcal{N}(r_i)|$ for each of the two types of nonlocal terms in \eqref{ch-12}.

We first consider the type~I of \eqref{ch-12}.  Define a tubular neighborhood
\begin{equation}\label{OMG}
\Omega_{\lambda^{-\frac13}}=\{x\in\Omega\,|\,\mathsf{d}(x)<\lambda^{-\frac13}\}\subset\Omega.    
\end{equation}
The standard Weyl's tube formula\footnote{(cf. \cite{W1939}) The volume of the tubular neighborhood $\Omega_{\mu}$ admits the following asymptotic expansion $$|\Omega_{\mu}| = \mu|\partial\Omega| - \frac{\mu^2}{2}\sum_{i=1}^{m-1}\int_{\partial\Omega} \kappa_i(x)\, \text{d}\mathcal{H}^{m-1} + O(\mu^3),\quad\text{as}\,\,\mu\to 0^+,$$ where  $ \kappa_i(x)$'s, $i=1,...,m-1$, are the principal curvatures at $x\in\partial\Omega$.} gives
\begin{equation}\label{OMV}
|\Omega_{\lambda^{-\frac13}}|=\lambda^{-\frac13}|\partial\Omega| +O(\lambda^{-\frac23})\quad\text{as}\,\,\lambda\gg1,
\end{equation}
where $|\partial\Omega|$ is the surface area of the boundary $\partial\Omega$.
 Using estimates~(\ref{ch-9})--(\ref{ch-10}), we can choose a sufficiently large constant~$\widetilde{\lambda}_i^*\geq\frac1{\mathfrak{M}_0^2}$ such that as $\lambda\geq\widetilde{\lambda}_i^*$,
\begin{equation}\label{ch-11}
    \begin{aligned}
        &\,|(r_i+\theta)\Phi_{\lambda}(x)+\Psi_{\lambda}(x)-r_i|=|(r_i+\theta)(\Phi_{\lambda}(x)-1)+\Psi_{\lambda}(x)+\theta|\\
  \leq&\,\left(|r_i|+\lambda^{\frac{1-p}{4p}}+\max_{\partial\Omega}|g|\right)\exp\left({-\frac{\lambda^{\frac16}\mathfrak{M}_0}2}\right)+\lambda^{\frac{1-p}{4p}}\\
\leq&\,\delta_{r_i},\,\,\text{for}\,\,x\in\Omega\setminus\Omega_{\lambda^{-\frac13}},
    \end{aligned}
\end{equation}
where $\delta_{r_i}$ is defined in \eqref{s0n}. In particular, \eqref{rrL} and \eqref{ch-11} indicate
\begin{equation}\label{ch-15}
\begin{aligned}
 &|\mathcal{N}((r_i+\theta)\Phi_{\lambda}(x)+\Psi_{\lambda}(x))-\mathcal{N}(r_i)|\\
 \leq&\,\mathcal{L}_{r_i}|(r_i+\theta)\Phi_{\lambda}(x)+\Psi_{\lambda}(x)-r_i|    \\
\leq &\,\mathcal{L}_{r_i}\left(\left(|r_i|+\lambda^{\frac{1-p}{4p}}+\max_{\partial\Omega}|g|\right)\exp\left({-\frac{\lambda^{\frac16}\mathfrak{M}_0}2}\right)+\lambda^{\frac{1-p}{4p}}\right).
\end{aligned}  
\end{equation}

Denote 
\begin{equation}\label{ch-16}
  \mathfrak{M}_{\lambda}(r_i):= \max_{|s|\leq|r_i|+\lambda^{\frac{1-p}{4p}}+\max_{\partial\Omega}|g|}\mathcal|{N}(s)|+|{N}(r_i)|.
\end{equation}
 Then, by (A2), \eqref{s0n}, \eqref{ch-9}--\eqref{ch-10}, \eqref{OMG} and \eqref{ch-15}, one may check that
 \begin{equation}\label{ch-13}
 \begin{aligned}
  &\left|\int_\Omega w(y)\left(\mathcal{N}((r_i+\theta)\Phi_{\lambda}(y)+\Psi_{\lambda}(y))-\mathcal{N}(r_i)\right)\dy\right|\\
=&\, \left|\left\{\int_{\Omega_{\lambda^{-\frac13}}} +\int_{\Omega\setminus\Omega_{\lambda^{-\frac13}}}\right\}w(y)\left(\mathcal{N}((r_i+\theta)\Phi_{\lambda}(y)+\Psi_{\lambda}(y))-\mathcal{N}(r_i)\right)\dy\right| \\
\leq&\,\mathfrak{M}_{\lambda}(r_i)\int_{\Omega_{\lambda^{-\frac13}}} |w(y)|\dy\\
&\,\,+\mathcal{L}_{r_i}\left(\left(|r_i|+\lambda^{\frac{1-p}{4p}}+\max_{\partial\Omega}|g|\right)\exp\left({-\frac{\lambda^{\frac16}\mathfrak{M}_0}2}\right)+\lambda^{\frac{1-p}{4p}}\right)\int_{\Omega\setminus\Omega_{\lambda^{-\frac13}}}|w(y)|\dy\\
\leq&\,\left(\mathfrak{M}_{\lambda}(r_i)|\Omega_{\lambda^{-\frac13}}|^{1-\frac1p}+\left(\widetilde{\mathfrak{M}}_{\lambda}(r_i)+\lambda^{\frac{1-p}{4p}}\mathcal{L}_{r_i}\right)|\Omega|^{1-\frac1p}\right)||w||_{\text{L}^{p}(\Omega)}
 \end{aligned}
 \end{equation}
where
\begin{equation}\label{ch-17}
\widetilde{\mathfrak{M}}_{\lambda}(r_i)   =\mathcal{L}_{r_i}\left(|r_i|+\lambda^{\frac{1-p}{4p}}+\max_{\partial\Omega}|g|\right)\exp\left({-\frac{\lambda^{\frac16}\mathfrak{M}_0}2}\right).
\end{equation}
 Note that $\mathcal{L}_{r_i}|\Omega|^{1-\frac1p}||w||_{\text{L}^{p}}<1$ (by \eqref{rrL}). Furthermore, since $p>1$, a direct calculation using \eqref{OMV}, \eqref{ch-16} and \eqref{ch-17} shows that
 $$\lambda^{\frac{p-1}{4p}}\left(\mathfrak{M}_{\lambda}(r_i)|\Omega_{\lambda^{-\frac13}}|^{1-\frac1p}+\widetilde{\mathfrak{M}}_{\lambda}(r_i)|\Omega|^{1-\frac1p}\right)||w||_{\text{L}^{p}(\Omega)}\xrightarrow{\lambda\to\infty}0.$$ 
 Hence, we can choose $\widetilde{\lambda}_{i,*}\geq\widetilde{\lambda}_i^*$ such that the following estimate holds for  $\lambda\geq\widetilde{\lambda}_{i,*}$:
 \begin{equation}\label{ch-18}
\left(\mathfrak{M}_{\lambda}(r_i)|\Omega_{\lambda^{-\frac13}}|^{1-\frac1p}+\widetilde{\mathfrak{M}}_{\lambda}(r_i)|\Omega|^{1-\frac1p}\right)||w||_{\text{L}^{p}(\Omega)}\leq\left(1-\mathcal{L}_{r_i}|\Omega|^{1-\frac1p}||w||_{\text{L}^{p}}\right)  \lambda^{\frac{1-p}{4p}} .  
 \end{equation}
Combining the results from \eqref{mapT}, \eqref{ch-12}, \eqref{ch-13} and \eqref{ch-18}, we obtain 
 \begin{equation*}
|\boldsymbol{\mathsf{T}}_{i,\lambda}(\theta)|=|\boldsymbol{\mathsf{N}}[(r_i+\theta)\Phi_{\lambda}+\Psi_{\lambda}]-\mathcal{N}(r_i)|
\leq\lambda^{\frac{1-p}{4p}},\quad\forall~\theta\in[-\lambda^{\frac{1-p}{4p}}, \lambda^{\frac{1-p}{4p}}].
 \end{equation*}
This completes the proof of \eqref{T*1} for type I.

It remains to deal with type II of \eqref{ch-12}. Note that $\int_\Omega w(y)\dy=1$ (by (A2)). Using estimates~(\ref{ch-9})--(\ref{ch-10}), one may check that 

\begin{equation}\label{ch-19}
{\begin{aligned}
     &\left| \int_\Omega w(y)((r_i+\theta)\Phi_{\lambda}(y)+\Psi_{\lambda}(y))\dy-r_i\right|\\
       \leq&\, \int_\Omega |w(y)|\left(|(r_i+\theta)(\Phi_{\lambda}(y)-1)|+|\Psi_{\lambda}(y)|+|\theta|\right)\dy\\      \leq&\,\int_\Omega|w(y)|\left(\left(|r_i|+\lambda^{\frac{1-p}{4p}}+\max_{\partial\Omega}|g|\right)\exp\left({-\frac{\mathsf{d}(y)}2\mathfrak{M}_0\sqrt{\lambda}}\right)+\lambda^{\frac{1-p}{4p}}\right)\dy\\
   \leq&\, ||w||_{\text{L}^{p}(\Omega)}\\
&\,\times\left(\left(|r_i|+\lambda^{\frac{1-p}{4p}}+\max_{\partial\Omega}|g|\right)\left(\int_\Omega\exp\left({-\frac{p\mathfrak{M}_0\sqrt{\lambda}\mathsf{d}(y)}{2(p-1)}}\right)\dy\right)^{1-\frac1p} +\lambda^{\frac{1-p}{4p}} |\Omega|^{1-\frac1p} \right).
    \end{aligned}}
\end{equation}
By a direct application of the coarea formula (cf. \cite[Theorems 3.10 and 3.14]{Evans2015}), we obtain the following estimate\footnote{We obtain \eqref{ch-20} by a similar argument to that for \cite[(43)]{LM2023}, and therefore omit the detailed proof.}:
\begin{equation}\label{ch-20}
 \frac1{|\Omega|}\int_\Omega\exp\left({-\frac{p\mathfrak{M}_0\sqrt{\lambda}\mathsf{d}(y)}{2(p-1)}}\right)\dy\leq \frac{\mathfrak{M}_2}{\sqrt{\lambda}},  
\end{equation}
where $\mathfrak{M}_2$ is a positive constant independent of $\lambda$. Since $-\frac12(1-\frac1p)<\frac{1-p}{4p}<0$,  \eqref{ch-20} gives
\begin{equation*}
\lambda^{\frac{p-1}{4p}}\left(|r_i|+\lambda^{\frac{1-p}{4p}}+\max_{\partial\Omega}|g|\right)\left(\int_\Omega\exp\left({-\frac{p\mathfrak{M}_0\sqrt{\lambda}\mathsf{d}(y)}{2(p-1)}}\right)\dy\right)^{1-\frac1p}\xrightarrow{\lambda\to\infty}0.    \end{equation*}
 In particular,  we can choose a sufficiently large constant~$\widetilde{\lambda}_i^*\geq\frac1{\mathfrak{M}_0^2}$ such that as $\lambda\geq\widetilde{\lambda}_i^*$, there holds
\begin{equation}\label{ch-23}
\begin{aligned}
   \left(|r_i|+\lambda^{\frac{1-p}{4p}}+\max_{\partial\Omega}|g|\right)&\left(\int_\Omega\exp\left({-\frac{p\mathfrak{M}_0\sqrt{\lambda}\mathsf{d}(y)}{2(p-1)}}\right)\dy\right)^{1-\frac1p}\\
&\qquad\qquad\leq\left(\frac1{\mathcal{L}_{r_i}||w||_{\text{L}^{p}(\Omega)}}-|\Omega|^{1-\frac1p}\right)\lambda^{\frac{1-p}{4p}}.  \end{aligned}
\end{equation}
Moreover, by \eqref{s0n}, \eqref{rrL}, \eqref{ch-19} and \eqref{ch-23}, we can choose $\widetilde{\lambda}_{i,*}\geq\widetilde{\lambda}_i^*$ such that $| \int_\Omega w(y)((r_i+\theta)\Phi_{\lambda}(y)+\Psi_{\lambda}(y))\dy-r_i|<\delta_{r_i}$ for  $\lambda\geq\widetilde{\lambda}_{i,*}$. This, combined with \eqref{ch-12}, leads to
 \begin{align*}
|\boldsymbol{\mathsf{T}}_{i,\lambda}(\theta)|=&\,|\boldsymbol{\mathsf{N}}[(r_i+\theta)\Phi_{\lambda}+\Psi_{\lambda}]-\mathcal{N}(r_i)|\\
=&\,\left|\mathcal{N}(\int_\Omega w(y)((r_i+\theta)\Phi_{\lambda}(y)+\Psi_{\lambda}(y))\dy)-\mathcal{N}(r_i)\right|\\
\leq&\,\mathcal{L}_{r_i}\left| \int_\Omega w(y)((r_i+\theta)\Phi_{\lambda}(y)+\Psi_{\lambda}(y))\dy-r_i\right|\\
\leq&\,\lambda^{\frac{1-p}{4p}},\quad\forall~\theta\in[-\lambda^{\frac{1-p}{4p}}, \lambda^{\frac{1-p}{4p}}].
 \end{align*}
Therefore, we prove \eqref{T*1} for type~II, and complete the proof of Lemma~\ref{lem1}.
\end{proof}

We will now show that for all sufficiently large $\lambda>0$ (possibly depending on $i$), the mapping~$\boldsymbol{\mathsf{T}}_{i,\lambda}$ defined by \eqref{mapT}  is a contraction from $[-\lambda^{\frac{1-p}{4p}}, \lambda^{\frac{1-p}{4p}}]$ into itself. As the arguments are similar to those provided in Lemma~\ref{lem1}, it is sufficient to present a detailed proof for the case of Type I of \eqref{ch-12} for clarity and rigor, i.e., 
\begin{equation}\label{mapTt}
 \boldsymbol{\mathsf{T}}_{i,\lambda}(\theta)= \int_\Omega w(y)\left(\mathcal{N}((r_i+\theta)\Phi_{\lambda}(y)+\Psi_{\lambda}(y))-\mathcal{N}(r_i)\right)\dy,\quad\theta\in[-\lambda^{\frac{1-p}{4p}}, \lambda^{\frac{1-p}{4p}}].    
 \end{equation}

 For a given $\lambda\geq\widetilde{\lambda}_{i,*}$ let $\theta_1,\theta_2\in[-\lambda^{\frac{1-p}{4p}}, \lambda^{\frac{1-p}{4p}}]$. By \eqref{ch-11}, we have
 \begin{equation*}
(r_i+\theta_1)\Phi_{\lambda}(y)+\Psi_{\lambda}(y),\,\,(r_i+\theta_2)\Phi_{\lambda}(y)+\Psi_{\lambda}(y)\in(r_i-\delta_{r_i},r_i+\delta_{r_i}),\,\,\forall y\in\Omega\setminus\Omega_{\lambda^{-\frac13}}.     
 \end{equation*}
 Consequently, applying \eqref{Phs}, \eqref{s0n}, \eqref{ch-9} and \eqref{OMG}, we obtain
\begin{equation}\label{ch-24}
    \begin{aligned}   &\int_{\Omega\setminus\Omega_{\lambda^{-\frac13}}} |w(y)|\left|\mathcal{N}((r_i+\theta_1)\Phi_{\lambda}(y)+\Psi_{\lambda}(y))-\mathcal{N}((r_i+\theta_2)\Phi_{\lambda}(y)+\Psi_{\lambda}(y))\right|\dy \\
  &\qquad\quad\quad  \leq\mathcal{L}_{r_i}|\theta_1-\theta_2|\int_{\Omega\setminus\Omega_{\lambda^{-\frac13}}}|w(y)|\Phi_{\lambda}(y)\dy
    \leq\mathcal{L}_{r_i}|\Omega|^{1-\frac1p}|w||_{\text{L}^{p}(\Omega)}|\theta_1-\theta_2|.
    \end{aligned}
\end{equation}
Here, a direct application of $0\leq\Phi_{\lambda}\leq1$ and H\"{o}lder's inequality yields the last estimate of \eqref{ch-24}. On the other hand, since 
$|(r_i+\theta_j)\Phi_{\lambda}(y)+\Psi_{\lambda}(y)|\leq|r_i|+\lambda^{\frac{1-p}{4p}}+ \max\limits_{\partial\Omega}|g|$ is uniformly bounded for $\lambda\geq\widetilde{\lambda}_{i,*}$,
by (A1), there exist finite numbers $\widetilde{L}_{r_i}>0$ and $\lambda_{i,*}\geq\widetilde{\lambda}_{i,*}$ such that, as $\lambda\geq\lambda_{i,*}$, we have 
 \begin{equation}\label{ch-25}
    \begin{aligned}   &\int_{\Omega_{\lambda^{-\frac13}}} |w(y)|\left|\mathcal{N}((r_i+\theta_1)\Phi_{\lambda}(y)+\Psi_{\lambda}(y))-\mathcal{N}((r_i+\theta_2)\Phi_{\lambda}(y)+\Psi_{\lambda}(y))\right|\dy \\
  &\qquad\quad  \leq\widetilde{L}_{r_i}|\theta_1-\theta_2|\int_{\Omega_{\lambda^{-\frac13}}}|w(y)|\Phi_{\lambda}(y)\dy
\leq\widetilde{L}_{r_i}|\Omega_{\lambda^{-\frac13}}|^{1-\frac1p}|w||_{\text{L}^{p}(\Omega)}|\theta_1-\theta_2|\\
 &\qquad\quad\leq\frac12\left(1-\mathcal{L}_{r_i}|\Omega|^{1-\frac1p}|w||_{\text{L}^{p}(\Omega)}\right)|\theta_1-\theta_2|.
 \end{aligned}
\end{equation}
The last estimate of \eqref{ch-25} is trivially valid due to $|\Omega_{\lambda^{-\frac13}}|\xrightarrow{\lambda\to\infty}0$. Combining \eqref{mapTt} with \eqref{ch-24}--\eqref{ch-25} yields
\begin{equation*}
    \begin{aligned}
  &| \boldsymbol{\mathsf{T}}_{i,\lambda}(\theta_1)- \boldsymbol{\mathsf{T}}_{i,\lambda}(\theta_2)| \\
  \leq&\,  \int_\Omega |w(y)|\left|\left(\mathcal{N}((r_i+\theta_1)\Phi_{\lambda}(y)+\Psi_{\lambda}(y))-\mathcal{N}((r_i+\theta_2)\Phi_{\lambda}(y)+\Psi_{\lambda}(y))\right)\right|\dy\\
\leq&\,\frac12\left(1+\mathcal{L}_{r_i}|\Omega|^{1-\frac1p}||w||_{\text{L}^{p}(\Omega)}\right)|\theta_1-\theta_2|,
    \end{aligned}
\end{equation*}
which is a contraction because of $\mathcal{L}_{r_i}|\Omega|^{1-\frac1p}|w||_{\text{L}^{p}(\Omega)}<1$. 
The contraction mapping theorem thus guarantees that for all~$\lambda\geq\lambda_{i,*}$, $\boldsymbol{\mathsf{T}}_{i,\lambda}:[-\lambda^{\frac{1-p}{4p}},\lambda^{\frac{1-p}{4p}}]\to[-\lambda^{\frac{1-p}{4p}},\lambda^{\frac{1-p}{4p}}]$ admits a unique fixed point $\theta_{i,\lambda}$. Therefore, the proof of Proposition~\ref{prop-t} is complete.
\subsection{\bf Proof of  Theorem~\ref{thm0}}\label{thmc}
We first consider the case where $\boldsymbol{\mathrm{S}}_{\star}=\{r_1,...,r_n\}$ is finite, consisting of exactly $n$ distinct elements. By Proposition~\ref{prop-t} we have $\boldsymbol{\mathsf{T}}_{i,\lambda}(\theta_{i,\lambda})=\theta_{i,\lambda}$ for some $|\theta_{i,\lambda}|\leq\lambda^{\frac{1-p}{4p}}$, as $\lambda\geq\lambda_{i,*}$. Together with \eqref{mapT}, this yields a solution~$u_{i,\lambda}=(r_i+\theta_{i,\lambda})\Phi_{\lambda}+\Psi_{\lambda}$ to equation~\eqref{eq2u} with the boundary condition~\eqref{rbd}, satisfying  $\boldsymbol{\mathsf{N}}[u_{i,\lambda}]=r_i+\theta_{i,\lambda}$, for each $i=1,...,n$. As a consequence, by \eqref{Phs} we have
\begin{equation}\label{bNu}
  \left| \boldsymbol{\mathsf{N}}[u_{i,\lambda}]-r_i \right|+\max_{\overline{\Omega}}  |u_{i,\lambda}-(r_i\Phi_{\lambda}+\Psi_{\lambda})|\leq2|\theta_{i,\lambda}|\leq2\lambda^{\frac{1-p}{4p}}\xrightarrow{\lambda\to\infty}0.
\end{equation}
From (A3), we see that all $r_i$  are independent of $\lambda$ and satisfy $r_i<r_{i+1}$, for $i\in\{1,...,n-1\}$. We can then find a sufficiently large $\lambda_{*}\geq\max\{\lambda_{i,*}|\,i=1,...,n\}$ such that
\begin{equation*}
\lambda_{*}^{\frac{1-p}{4p}}<\min\limits_{1\leq i\leq n-1}\frac{r_{i+1}-r_i}2.    
\end{equation*}
 This guarantees $r_i+\theta_{i,\lambda}<r_{i+1}+\theta_{i+1,\lambda}$ for $i=1,...,n-1$, as $\lambda\geq\lambda_{*}$.

 On the other hand, by \eqref{dhb} we see that the solution $\Phi_{\lambda}$ of \eqref{eq-Phi} is positive in $\Omega$. As a direct result, we immediately obtain 
\begin{equation*}
u_{i+1,\lambda}-u_{i,\lambda}=[(r_{i+1}+\theta_{i+1,\lambda})-(r_i+\theta_{i,\lambda})]\Phi_{\lambda}>0\quad\text{in}\,\, \Omega.   
\end{equation*}
 This shows that all solutions~$u_{i,\lambda}$ satisfying  \eqref{2.3u} are distinct when $\lambda\geq\lambda_{*}$. Furthermore, \eqref{u-1as} follows directly from \eqref{bNu}.

We now turn to the case where $b(x)>0$ on $\partial\Omega$, and claim that $\min_{\overline{\Omega}}\Phi_{\lambda}>0$ for all $\lambda>0$. This consequence is a direct application of Hopf lemma (cf. \cite[Theorem~2.7 \& Corollary~2.9]{HL1997}). To prove this, we set ${\phi}_{\lambda}=1-\Phi_{\lambda}$. Applying the maximum principle to the equation of ${\phi}_{\lambda}$ with the condition~\eqref{dhb}, we find that ${\phi}_{\lambda}$ must attain its maximum value at a boundary point~$x_{\text{b}}\in\partial\Omega$. Since the smoothness of $\partial\Omega$ ensures that $\Omega$ has the interior sphere property,  it follows from \cite[Corollary~2.9]{HL1997} that  $\frac{\partial{\phi}_{\lambda}}{\partial{\vec{\nu}}}(x_{\text{b}})>0$. This is equivalent to
\begin{equation}\label{maxP}
\min_{\overline{\Omega}}\Phi_{\lambda}=\Phi_{\lambda}(x_{\text{b}})=-b(x_{\text{b}})\frac{\partial\Phi_{\lambda}}{\partial{\vec{\nu}}}(x_{\text{b}})=b(x_{\text{b}})\frac{\partial{\phi}_{\lambda}}{\partial{\vec{\nu}}}(x_{\text{b}})>0. 
\end{equation}
  Consequently, the combination of \eqref{2.3u} and \eqref{maxP} implies $u_{i+1,\lambda}-u_{i,\lambda}=[(r_{i+1}+\theta_{i+1,\lambda})-(r_i+\theta_{i,\lambda})]\Phi_{\lambda}>0$ on $\overline{\Omega}$. This completes the proof of Theorem~\ref{thm0}(i).

Theorem~\ref{thm0}(ii) is a direct consequence of Theorem~\ref{thm0}(i). By considering a subset $\{r_1,...,r_{\widetilde{n}}\}$ of $\boldsymbol{\mathrm{S}}_{\star}$ such that $r_1<\cdots<r_{\widetilde{n}}$ and following the same argument as in Theorem~\ref{thm0}(i), we obtain that for each $\lambda_{\widetilde{n}}>0$ such that as $\lambda\geq\lambda_{\widetilde{n}}$ (for some $\lambda_{\widetilde{n}}>0$), the equation possesses at least $\widetilde{n}$ distinct 
 solutions of the form $(r_i+\theta_{i,\lambda})\Phi_{\lambda}+\Psi_{\lambda}$, for $i=1,...,\widetilde{n}$. Consequently, the proof of Theorem~\ref{thm0} is complete.

 \section{\bf Proof of  Theorem~\ref{thm2}}\label{sec-thm2}
 \subsection{Asymptotic estimates of $\Phi_{j,\lambda,\eta}$ and $\Psi_{0,\lambda}$}

We first establish the asymptotic estimate  of $\Phi_{j,\lambda,\eta}$ and $\Psi_{0,\lambda}$ with respect to the parameters $\lambda$ and $\eta$.
 By \eqref{dhb} and \eqref{HN}, we can apply the standard maximum principle to \eqref{eq-Phj} and obtain an upper bound of $|\Phi_{j,\lambda,\eta}|$: 
 \begin{equation}\label{mPj}
\max_{\overline{\Omega}}|\Phi_{j,\lambda,\eta}|\leq\frac{\eta}{\lambda}\max_{\overline{\Omega}}\frac{|h_j|}{f}\leq\frac{\eta}{\lambda}||\overrightarrow{\boldsymbol{\mathrm{H}}}||_{\mathrm{L}^{\infty}(\Omega)}.
 \end{equation}
Rearranging the expression in \eqref{eq-Phj}, we define a new function $\widetilde{\Phi}_j$ for simplicity by setting 
\begin{equation}\label{ch-pj}
\widetilde{\Phi}_j:=\Phi_{j,\lambda,\eta}-\frac{\eta h_j(x)}{\lambda f(x)}. 
\end{equation}
Under \eqref{ch-pj} and assumption (A4), we obtain equation of $\widetilde{\Phi}_j$:
\begin{equation}\label{wPj}
\frac1{\lambda}\nabla\cdot(\od(x)\nabla \widetilde{\Phi}_j) =f(x) \widetilde{\Phi}_j-\frac{\eta}{\lambda^2}\nabla\cdot\left(\od(x)\nabla \frac{h_j(x)}{f(x)}\right).
\end{equation}
For convenience, we omit the subscripts $\lambda$ and $\eta$ on $\widetilde{\Phi}_j$ as we will only perform estimates in the subsequent analysis. Multiplying \eqref{wPj} by $\widetilde{\Phi}_j$ and using the uniform ellipticity condition (see \eqref{dhb}-(a)), one may check that
\begin{equation}\label{6-ch}
    \begin{aligned}
&\frac1{\lambda}\nabla\cdot\left(\od(x)\nabla \left(\widetilde{\Phi}_j^2-\frac{\eta^2}{\lambda^4}\mathfrak{M}_3^2\right)\right)\\
\geq&\,2\mathfrak{D}|\nabla\widetilde{\Phi}_j|^2+2f(x) \widetilde{\Phi}_j^2-\frac{2\eta}{\lambda^2}\widetilde{\Phi}_j\nabla\cdot\left(\od(x)\nabla \frac{h_j(x)}{f(x)}\right)\\
 \geq&\,f(x) \left(\widetilde{\Phi}_j^2-\frac{\eta^2}{\lambda^4}\mathfrak{M}_{3,j}^2\right)\quad\text{in}\,\,\Omega, \\
    \end{aligned}
\end{equation}
where we define 
\begin{equation*}
\mathfrak{M}_{3,j}=\max_{\overline{\Omega}}\frac1{f}\left|\nabla\cdot\left(\od\nabla \frac{h_j}{f}\right)\right|,\quad j=0,1,...,k,    
\end{equation*}
which is a positive constant independent of $\lambda$. Then by \eqref{HN}, \eqref{mPj} and \eqref{6-ch}, we can employ the same argument as in \eqref{ch-8}--\eqref{ch-10} to obtain
\begin{equation*}
    \widetilde{\Phi}_j^2(x)-\frac{\eta^2}{\lambda^4}\mathfrak{M}_{3,j}^2\leq\frac{\eta^2}{\lambda^2}\left(4||\overrightarrow{\boldsymbol{\mathrm{H}}}||_{\mathrm{L}^{\infty}(\Omega)}^2+ \frac{\mathfrak{M}_{3,j}^2}{\lambda^2}\right)\exp\left({-\mathsf{d}(x)\mathfrak{M}_0\sqrt{\lambda}}\right)\quad\text{in}\,\,\overline{\Omega},
\end{equation*}
as $\lambda\geq\frac{1}{\mathfrak{M}_0^2}$, where $\mathfrak{M}_0$ is defined in \eqref{ml0}. Here the right-hand side of the above estimate is simplified by once again using the estimate  $\max_{\overline{\Omega}}\frac{|h_j|}{f}\leq||\overrightarrow{\boldsymbol{\mathrm{H}}}||_{\mathrm{L}^{\infty}(\Omega)}$.  As a consequence, for $j=1,...,k$, we arrive at
\begin{equation}\label{8-ch}
    \begin{aligned}
&\left|\Phi_{j,\lambda,\eta}(x)-\frac{h_j(x)}{ f(x)}\right|\\
\leq&\,\frac{|\eta-\lambda|}{\lambda}\max_{\overline{\Omega}}\frac{|h_j|}{f}+| \widetilde{\Phi}_j(x)|\\
\leq&\,\frac{|\eta-\lambda|}{\lambda}||\overrightarrow{\boldsymbol{\mathrm{H}}}||_{\mathrm{L}^{\infty}(\Omega)}+\frac{\eta}{\lambda^2}\mathfrak{M}_{3,j}\\
&\,+\frac{\eta}{\lambda}\left(2||\overrightarrow{\boldsymbol{\mathrm{H}}}||_{\mathrm{L}^{\infty}(\Omega)}+ \frac{\mathfrak{M}_{3,j}}{\lambda}\right)\exp\left({-\frac{\mathsf{d}(x)}2\mathfrak{M}_0\sqrt{\lambda}}\right)\quad\text{in}\,\,\overline{\Omega},
    \end{aligned}
\end{equation}
as $\lambda\geq\frac{1}{\mathfrak{M}_0^2}$.

The solution $\Psi_{0,\lambda}$ to \eqref{eq-Ps0} satisfies the maximum bound:
\begin{equation}\label{p0bd} \max_{\overline{\Omega}}|\Psi_{0,\lambda}|\leq\max\left\{\max_{\partial\Omega}|g|,\frac1{\lambda}\max_{\overline{\Omega}}\frac{|h_0|}{f}\right\}.
\end{equation}
We then proceed by applying the same argument as in \eqref{6-ch} to $\widetilde{\Psi}_0(x)=\Psi_{0,\lambda}(x)-\frac{h_0(x)}{\lambda f(x)}$, which yields
\begin{equation}\label{9-ch}
\begin{aligned}
\frac1{\lambda}\nabla\cdot(\od(x)\nabla \widetilde{\Psi}_0^2)
\geq&\, f(x) \left(\widetilde{\Psi}_0^2-\frac{1}{\lambda^4f^2}\left|\nabla\cdot\left(\od(x)\nabla \frac{h_0(x)}{f(x)}\right)\right|^2\right)\\
\geq&\, f(x) \left(\widetilde{\Psi}_0^2-\frac{\mathfrak{M}_{3,0}^2}{\lambda^4}\right)
\quad\text{in}\,\,\Omega.
\end{aligned}
\end{equation}
Let us consider the case where $g\neq0$ on $\partial\Omega$. Hence, by applying the same argument as in \eqref{6-ch}--\eqref{8-ch} to \eqref{9-ch} and utilizing \eqref{p0bd}, we arrive at the following estimate for $\lambda\geq\max\{\frac{1}{\mathfrak{M}_0^2},\frac{\max_{\overline{\Omega}}\frac{|h_0|}{f}}{\max_{\partial\Omega}|g|}\}$:
\begin{equation}\label{10-ch}
    |\Psi_{0,\lambda}(x)|\leq\frac1{\lambda}\max_{\overline{\Omega}}\frac{|h_0|}{f}+\max_{\partial\Omega}|g|\exp\left({-\frac{\mathsf{d}(x)}2\mathfrak{M}_0\sqrt{\lambda}}\right),\quad\forall\,x\in\overline{\Omega}.
\end{equation}
When $g\equiv0$ on $\partial\Omega$, \eqref{10-ch} holds for $\lambda>0$, trivially due to \eqref{p0bd}.

\subsection{Completion of the proof of  Theorem~\ref{thm2}}\label{finpf}
To prove Theorem~\ref{thm2}, we first consider the mapping $\boldsymbol{\Gamma}_{\vec{\boldsymbol{\mathrm{r}}},\lambda,\eta}:{\boldsymbol{\mathrm{Q}}_{\lambda}}\to{\boldsymbol{\mathrm{Q}}_{\lambda}}$ defined in \eqref{nT}. Here, we define the closed and complete metric space 
\begin{equation*}
  {\boldsymbol{\mathrm{Q}}_{\lambda}}:=\{\vec{\boldsymbol{\vartheta}}\in\mathbb{R}^k|\,\|\vec{\boldsymbol{\vartheta}}\|_{\ell^\infty}\leq\frac{\lambda^{\frac{1-p}{4p}}}{k}\} 
\end{equation*}
 equipped with the standard maximum norm $\|\cdot\|_{\ell^\infty}$ on $\mathbb{R}^k$. We then introduce the following property:
\begin{proposition}\label{prop-3}
Under the same hypotheses as in Theorem~\ref{thm2}, for each $\vec{\boldsymbol{\mathrm{r}}}\in\boldsymbol{\mathrm{S}}_k$, there exists a positive constant $\lambda_{\vec{\boldsymbol{\mathrm{r}}},L,\tau}^*$ depending mainly on $\vec{\boldsymbol{\mathrm{r}}}$, $L$ and $\tau$ such that for all $\lambda\geq\lambda_{\vec{\boldsymbol{\mathrm{r}}},L,\tau}^*$ and $\eta$ satisfying \eqref{le-*i}, the mapping $\boldsymbol{\Gamma}_{\vec{\boldsymbol{\mathrm{r}}},\lambda,\eta}:{\boldsymbol{\mathrm{Q}}_{\lambda}}\to{\boldsymbol{\mathrm{Q}}_{\lambda}}$  defined in \eqref{nT} admits a unique fixed point $\vec{\boldsymbol{\vartheta}}_{\vec{\boldsymbol{\mathrm{r}}},\lambda,\eta}$.    
\end{proposition}
\begin{proof}
For $\vec{\boldsymbol{\mathrm{r}}}\in\boldsymbol{\mathrm{S}}_k$ and  $\vec{\boldsymbol{\vartheta}}\in{\boldsymbol{\mathrm{Q}}_{\lambda}}$, it follows from \eqref{rNH}, \eqref{nT} and \eqref{vP} that
\begin{equation}\label{ch-T}
\|\boldsymbol{\Gamma}_{\vec{\boldsymbol{\mathrm{r}}},\lambda,\eta}(\vec{\boldsymbol{\vartheta}})\|_{\ell^\infty}\leq\max_{1\leq j\leq k}\left|\boldsymbol{\mathsf{N}}_j\big[(\vec{\boldsymbol{\mathrm{r}}}+\vec{\boldsymbol{\vartheta}})\boldsymbol{\cdot}\overrightarrow{\boldsymbol{\Phi}}_{\lambda,\eta}+\Psi_{0,\lambda}\big]-\boldsymbol{\mathsf{N}}_j\big[\vec{\boldsymbol{\mathrm{r}}}\boldsymbol{\cdot}\overrightarrow{\boldsymbol{\mathrm{H}}}\big]\right|.
\end{equation}
Based on \eqref{N-2j} and the property that $\int_{\Omega}w_j(y)\dy=1$, the expression related to the right-hand side of \eqref{ch-T} takes one of the following two forms:
\begin{equation}\label{N-3j}
\begin{aligned}
&\boldsymbol{\mathsf{N}}_j\big[(\vec{\boldsymbol{\mathrm{r}}}+\vec{\boldsymbol{\vartheta}})\boldsymbol{\cdot}\overrightarrow{\boldsymbol{\Phi}}_{\lambda,\eta}+\Psi_{0,\lambda}\big]-\boldsymbol{\mathsf{N}}_j\big[\vec{\boldsymbol{\mathrm{r}}}\boldsymbol{\cdot}\overrightarrow{\boldsymbol{\mathrm{H}}}\big]\\
=&\, \begin{cases}
\displaystyle\int_\Omega w_j(y)\left(\mathcal{N}_j((\vec{\boldsymbol{\mathrm{r}}}+\vec{\boldsymbol{\vartheta}})\boldsymbol{\cdot}\overrightarrow{\boldsymbol{\Phi}}_{\lambda,\eta}(y)+\Psi_{0,\lambda}(y))-\mathcal{N}_j(\vec{\boldsymbol{\mathrm{r}}}\boldsymbol{\cdot}\overrightarrow{\boldsymbol{\mathrm{H}}})\right)\dy & (\text{type~I});\vspace{3pt} \\
\displaystyle\mathcal{N}_j(\int_\Omega w_j(y)(\vec{\boldsymbol{\mathrm{r}}}+\vec{\boldsymbol{\vartheta}})\boldsymbol{\cdot}\overrightarrow{\boldsymbol{\Phi}}_{\lambda,\eta}(y)+\Psi_{0,\lambda}(y)\dy)-\mathcal{N}_j(\vec{\boldsymbol{\mathrm{r}}}\boldsymbol{\cdot}\overrightarrow{\boldsymbol{\mathrm{H}}}) & (\text{type~II}).
\end{cases}
\end{aligned}
\end{equation}
For the sake of completeness, we will now estimate the right-hand side of \eqref{ch-T} for each of the two types in \eqref{N-3j}. We begin with the estimate for the first type. Without loss of generality, we only consider the case $g\not\equiv0$ on $\partial\Omega$. Let us fix an arbitrary constant $\widehat{L}_1>1$. We require that $\lambda$ and $\eta$ satisfy
\begin{equation}\label{Leta}
\lambda\geq\max\{\frac{1}{\mathfrak{M}_0^2},\frac{\max_{\overline{\Omega}}\frac{|h_0|}{f}}{\max_{\partial\Omega}|g|},\max_{\overline{\Omega}}\frac{|h_0|}{f}+1\},\quad\frac{\eta}{\lambda}\leq\widehat{L}_1.
\end{equation}
When $\lambda>0$ is sufficiently large, we have the asymptotic relation $L\lambda^{\tau+\frac{1+3p}{4p}}\ll(\widehat{L}_1-1)\lambda$, which implies that any $\eta$ satisfying \eqref{le-*i} will also satisfy \eqref{Leta}.

By (\ref{vP}), (\ref{mPj}), (\ref{p0bd})  and \eqref{Leta}, we have
\begin{equation}\label{11-ch}
\begin{aligned}
  \max_{\overline{\Omega}}&\left|(\vec{\boldsymbol{\mathrm{r}}}+\vec{\boldsymbol{\vartheta}})\boldsymbol{\cdot}\overrightarrow{\boldsymbol{\Phi}}_{\lambda,\eta}+\Psi_{0,\lambda}\right|\\
  \leq&\,\frac{k\eta}{\lambda}\left(||\vec{\boldsymbol{\mathrm{r}}}||_{\ell^\infty}+\frac{\lambda^{\frac{1-p}{4p}}}k\right)||\overrightarrow{\boldsymbol{\mathrm{H}}}||_{\mathrm{L}^{\infty}(\Omega)}+\max\left\{\max_{\partial\Omega}|g|,\frac1{\lambda}\max_{\overline{\Omega}}\frac{|h_0|}{f}\right\}  \\
\leq&\,\widehat{L}_1(k||\vec{\boldsymbol{\mathrm{r}}}||_{\ell^\infty}+1)||\overrightarrow{\boldsymbol{\mathrm{H}}}||_{\mathrm{L}^{\infty}(\Omega)}+ \max_{\partial\Omega}|g|+1:=\widehat{\mathfrak{M}}_{\vec{\boldsymbol{\mathrm{r}}}},\quad\text{for}~\vec{\boldsymbol{\mathrm{r}}}\in\boldsymbol{\mathrm{S}}_k. 
\end{aligned}
\end{equation}
On the other hand, by applying \eqref{vP}, \eqref{OMG}, \eqref{8-ch},  \eqref{10-ch} and \eqref{Leta}, we can establish the following interior estimate for $\lambda$ and $\frac{\eta}{\lambda}$ satisfying \eqref{Leta}:
\begin{equation}\label{12-ch}
{\small\begin{aligned}
 & \max_{\overline{\Omega}\setminus\Omega_{\lambda^{-\frac13}}}\left|\left((\vec{\boldsymbol{\mathrm{r}}}+\vec{\boldsymbol{\vartheta}})\boldsymbol{\cdot}\overrightarrow{\boldsymbol{\Phi}}_{\lambda,\eta}+\Psi_{0,\lambda}\right)-\vec{\boldsymbol{\mathrm{r}}}\boldsymbol{\cdot}\overrightarrow{\boldsymbol{\mathrm{H}}}\right|\\
=&\,\max_{\overline{\Omega}\setminus\Omega_{\lambda^{-\frac13}}}\left|(\vec{\boldsymbol{\mathrm{r}}}+\vec{\boldsymbol{\vartheta}})\boldsymbol{\cdot}\left(\overrightarrow{\boldsymbol{\Phi}}_{\lambda,\eta}-\overrightarrow{\boldsymbol{\mathrm{H}}}\right)+\vec{\boldsymbol{\vartheta}}\boldsymbol{\cdot}\overrightarrow{\boldsymbol{\mathrm{H}}}+\Psi_{0,\lambda}\right|\\
\leq&\,k\left(\left(||\vec{\boldsymbol{\mathrm{r}}}||_{\ell^\infty}+\frac{\lambda^{\frac{1-p}{4p}}}k\right)||\overrightarrow{\boldsymbol{\Phi}}_{\lambda,\eta}-\overrightarrow{\boldsymbol{\mathrm{H}}}||_{\mathrm{L}^{\infty}(\overline{\Omega}\setminus\Omega_{\lambda^{-\frac13}})}+\frac{\lambda^{\frac{1-p}{4p}}}k||\overrightarrow{\boldsymbol{\mathrm{H}}}||_{\mathrm{L}^{\infty}(\Omega)}\right)+\max_{\overline{\Omega}\setminus\Omega_{\lambda^{-\frac13}}}\left|\Psi_{0,\lambda}\right|\\
\leq&\,\left(k||\vec{\boldsymbol{\mathrm{r}}}||_{\ell^\infty}+{\lambda^{\frac{1-p}{4p}}}\right)\\
&\,\times\left(\frac{|\eta-\lambda|||\overrightarrow{\boldsymbol{\mathrm{H}}}||_{\mathrm{L}^{\infty}(\Omega)}+\widehat{L}_1\mathfrak{M}_{3,j}}{\lambda}+\widehat{L}_1\left(2||\overrightarrow{\boldsymbol{\mathrm{H}}}||_{\mathrm{L}^{\infty}(\Omega)}+ \frac{\mathfrak{M}_{3,j}}{\lambda}\right)\exp\left({-\frac{\mathfrak{M}_0\lambda^{\frac16}}2}\right)\right)\\
&\,+{\lambda^{\frac{1-p}{4p}}}||\overrightarrow{\boldsymbol{\mathrm{H}}}||_{\mathrm{L}^{\infty}(\Omega)}+\frac1{\lambda}\max_{\overline{\Omega}}\frac{|h_0|}{f}+\max_{\partial\Omega}|g|\exp\left({-\frac{\mathfrak{M}_0\lambda^{\frac16}}2}\right)\\
=&\,\left(\left(k||\vec{\boldsymbol{\mathrm{r}}}||_{\ell^\infty}+{\lambda^{\frac{1-p}{4p}}}\right)\frac{|\eta-\lambda|}{\lambda}+{\lambda^{\frac{1-p}{4p}}}\right)||\overrightarrow{\boldsymbol{\mathrm{H}}}||_{\mathrm{L}^{\infty}(\Omega)}\\
&\,+\frac1{\lambda}\left(\widehat{L}_1\mathfrak{M}_{3,j}\left(k||\vec{\boldsymbol{\mathrm{r}}}||_{\ell^\infty}+{\lambda^{\frac{1-p}{4p}}}\right)+\max_{\overline{\Omega}}\frac{|h_0|}{f}\right)\\
&\,+\left(\widehat{L}_1\left(k||\vec{\boldsymbol{\mathrm{r}}}||_{\ell^\infty}+{\lambda^{\frac{1-p}{4p}}}\right)\left(2||\overrightarrow{\boldsymbol{\mathrm{H}}}||_{\mathrm{L}^{\infty}(\Omega)}+ \frac{\mathfrak{M}_{3,j}}{\lambda}\right)+\max_{\partial\Omega}|g|\right)\exp\left({-\frac{\mathfrak{M}_0\lambda^{\frac16}}2}\right).
\end{aligned}}
\end{equation}
Our next argument relies on the condition that $\lambda$ and $\eta$ satisfying \eqref{le-*i}. Note that $\tau<0$ so we have $\frac{|\eta-\lambda|}{\lambda}\leq L\lambda^{\tau+\frac{1-p}{4p}}\xrightarrow{\lambda\to\infty}0$.  Under \eqref{le-*i}, it follows from \eqref{Leta} and \eqref{12-ch} that for $\lambda\geq\widehat{\lambda}_{\vec{\boldsymbol{\mathrm{r}}},L,\tau}$,  
\begin{equation}\label{d-L2}
\max_{\overline{\Omega}\setminus\Omega_{\lambda^{-\frac13}}}\left|\left((\vec{\boldsymbol{\mathrm{r}}}+\vec{\boldsymbol{\vartheta}})\boldsymbol{\cdot}\overrightarrow{\boldsymbol{\Phi}}_{\lambda,\eta}+\Psi_{0,\lambda}\right)-\vec{\boldsymbol{\mathrm{r}}}\boldsymbol{\cdot}\overrightarrow{\boldsymbol{\mathrm{H}}}\right|<\delta_{\vec{\boldsymbol{\mathrm{r}}}},
\end{equation}
 where $\delta_{\vec{\boldsymbol{\mathrm{r}}}}$ is defined in \eqref{sk} and $\widehat{\lambda}_{\vec{\boldsymbol{\mathrm{r}}},L,\tau}$ depending on $\vec{\boldsymbol{\mathrm{r}}}\in\boldsymbol{\mathrm{S}}_k$, $L$ and $\tau$ is a constant satisfying
 \begin{equation*}
\widehat{\lambda}_{\vec{\boldsymbol{\mathrm{r}}},L,\tau}\geq\max\{\frac{1}{\mathfrak{M}_0^2},\frac{\max_{\overline{\Omega}}\frac{|h_0|}{f}}{\max_{\partial\Omega}|g|},\max_{\overline{\Omega}}\frac{|h_0|}{f}+1,\left(\frac{L}{\widehat{L}_1-1}\right)^{\frac1{-\tau+\frac{p-1}{4p}}}\}.   
 \end{equation*}
 This result, together with \eqref{sk}, yields that for all $j=1,...,k$,
\begin{equation}\label{skn}
\begin{aligned}
&\max_{\overline{\Omega}\setminus\Omega_{\lambda^{-\frac13}}}|\mathcal{N}_j((\vec{\boldsymbol{\mathrm{r}}}+\vec{\boldsymbol{\vartheta}})\boldsymbol{\cdot}\overrightarrow{\boldsymbol{\Phi}}_{\lambda,\eta}+\Psi_{0,\lambda})-\mathcal{N}_j(\vec{\boldsymbol{\mathrm{r}}}\boldsymbol{\cdot}\overrightarrow{\boldsymbol{\mathrm{H}}})|\\
\leq&\,\mathcal{L}_{\vec{\boldsymbol{\mathrm{r}}}}\max_{\overline{\Omega}\setminus\Omega_{\lambda^{-\frac13}}}|(\vec{\boldsymbol{\mathrm{r}}}+\vec{\boldsymbol{\vartheta}})\boldsymbol{\cdot}\overrightarrow{\boldsymbol{\Phi}}_{\lambda,\eta}+\Psi_{0,\lambda}-\vec{\boldsymbol{\mathrm{r}}}\boldsymbol{\cdot}\overrightarrow{\boldsymbol{\mathrm{H}}}|,\qquad\text{as}\,\,\lambda\geq\widehat{\lambda}_{\vec{\boldsymbol{\mathrm{r}}},L,\tau}.
\end{aligned}
\end{equation}

\eqref{11-ch} and \eqref{d-L2}--\eqref{skn} provide the motivation for splitting the integral of Type~I in \eqref{N-3j} into two separate integrals over domains $\Omega_{\lambda^{-\frac13}}$ and $\overline{\Omega}\setminus\Omega_{\lambda^{-\frac13}}$. We then estimate these integrals individually. First, by applying \eqref{OMV}, the estimate~\eqref{11-ch} and H\"{o}lder's inequality, we obtain 
\begin{equation}\label{1ty}
    \begin{aligned}
&\left|\int_{\Omega_{\lambda^{-\frac13}}}w_j(y)\left(\mathcal{N}_j((\vec{\boldsymbol{\mathrm{r}}}+\vec{\boldsymbol{\vartheta}})\boldsymbol{\cdot}\overrightarrow{\boldsymbol{\Phi}}_{\lambda,\eta}(y)+\Psi_{0,\lambda}(y))-\mathcal{N}_j(\vec{\boldsymbol{\mathrm{r}}}\boldsymbol{\cdot}\overrightarrow{\boldsymbol{\mathrm{H}}})\right)\dy\right| \\
\leq&\left(\max_{|s|\leq\widehat{\mathfrak{M}}_{\vec{\boldsymbol{\mathrm{r}}}}}|\mathcal{N}_j(s)|+|\mathcal{N}_j(\vec{\boldsymbol{\mathrm{r}}}\boldsymbol{\cdot}\overrightarrow{\boldsymbol{\mathrm{H}}})|\right)|\Omega_{\lambda^{-\frac13}}|^{1-\frac1p}||w_j||_{\text{L}^{p}(\Omega)}\\
=&\,\lambda^{\frac{1-p}{3p}}\left(\max_{|s|\leq\widehat{\mathfrak{M}}_{\vec{\boldsymbol{\mathrm{r}}}}}|\mathcal{N}_j(s)|+|\mathcal{N}_j(\vec{\boldsymbol{\mathrm{r}}}\boldsymbol{\cdot}\overrightarrow{\boldsymbol{\mathrm{H}}})|\right)|\partial\Omega|^{1-\frac1p}\left(||w_j||_{\text{L}^{p}(\Omega)}+\boldsymbol{\mathrm{o}}_{\lambda}\right),
    \end{aligned}
\end{equation}
where $\boldsymbol{\mathrm{o}}_{\lambda}$ denotes a quantity tending to zero as $\lambda\to\infty$. 
On the other hand, by \eqref{le-*i}, \eqref{12-ch} and \eqref{skn}, one may check that
\begin{equation}\label{2ty}
    \begin{aligned}
&\left|\int_{\Omega\setminus\overline{\Omega_{\lambda^{-\frac13}}}}w_j(y)\left(\mathcal{N}_j((\vec{\boldsymbol{\mathrm{r}}}+\vec{\boldsymbol{\vartheta}})\boldsymbol{\cdot}\overrightarrow{\boldsymbol{\Phi}}_{\lambda,\eta}(y)+\Psi_{0,\lambda}(y))-\mathcal{N}_j(\vec{\boldsymbol{\mathrm{r}}}\boldsymbol{\cdot}\overrightarrow{\boldsymbol{\mathrm{H}}})\right)\dy\right| \\
\leq&\,\mathcal{L}_{\vec{\boldsymbol{\mathrm{r}}}}|\Omega|^{1-\frac1p}||w_j||_{\text{L}^{p}(\Omega)}\max_{\overline{\Omega}\setminus\Omega_{\lambda^{-\frac13}}}|(\vec{\boldsymbol{\mathrm{r}}}+\vec{\boldsymbol{\vartheta}})\boldsymbol{\cdot}\overrightarrow{\boldsymbol{\Phi}}_{\lambda,\eta}+\Psi_{0,\lambda}-\vec{\boldsymbol{\mathrm{r}}}\boldsymbol{\cdot}\overrightarrow{\boldsymbol{\mathrm{H}}}|\\
\leq&\,\mathcal{L}_{\vec{\boldsymbol{\mathrm{r}}}}|\Omega|^{1-\frac1p}||w_j||_{\text{L}^{p}(\Omega)}\left(\left(\left(k||\vec{\boldsymbol{\mathrm{r}}}||_{\ell^\infty}+{\lambda^{\frac{1-p}{4p}}}\right)L\lambda^{\tau}+1\right)||\overrightarrow{\boldsymbol{\mathrm{H}}}||_{\mathrm{L}^{\infty}(\Omega)}+\boldsymbol{\mathrm{o}}_{\lambda}\right){\lambda^{\frac{1-p}{4p}}}\\
=&\,\mathcal{L}_{\vec{\boldsymbol{\mathrm{r}}}}|\Omega|^{1-\frac1p}||w_j||_{\text{L}^{p}(\Omega)}\left(\left(kL||\vec{\boldsymbol{\mathrm{r}}}||_{\ell^\infty}\lambda^{\tau}+1\right)||\overrightarrow{\boldsymbol{\mathrm{H}}}||_{\mathrm{L}^{\infty}(\Omega)}+\boldsymbol{\mathrm{o}}_{\lambda}\right){\lambda^{\frac{1-p}{4p}}}.
    \end{aligned}
\end{equation}
The last two estimates of \eqref{2ty} are simplified by applying the trivial asymptotics $\exp\left({-\frac{\mathfrak{M}_0\lambda^{\frac16}}2}\right)\ll\frac1\lambda\ll\lambda^{\frac{1-p}{4p}}\ll1$.

Hence, by applying the estimates \eqref{1ty} and \eqref{2ty} to the Type I integral of \eqref{N-3j}, we arrive at the following refined estimate of \eqref{ch-T} for each $\vec{\boldsymbol{\mathrm{r}}}\in\boldsymbol{\mathrm{S}}_k$ and $\vec{\boldsymbol{\vartheta}}\in{\boldsymbol{\mathrm{Q}}_{\lambda}}$:
\begin{equation}\label{ch2-T}
\begin{aligned}
&\|\boldsymbol{\Gamma}_{\vec{\boldsymbol{\mathrm{r}}},\lambda,\eta}(\vec{\boldsymbol{\vartheta}})\|_{\ell^\infty}\\
\leq&\,\max_{1\leq j\leq k}\left|\boldsymbol{\mathsf{N}}_j\big[(\vec{\boldsymbol{\mathrm{r}}}+\vec{\boldsymbol{\vartheta}})\boldsymbol{\cdot}\overrightarrow{\boldsymbol{\Phi}}_{\lambda,\eta}+\Psi_{0,\lambda}\big]-\boldsymbol{\mathsf{N}}_j\big[\vec{\boldsymbol{\mathrm{r}}}\boldsymbol{\cdot}\overrightarrow{\boldsymbol{\mathrm{H}}}\big]\right|\\
\leq&\,\lambda^{\frac{1-p}{3p}}\left(\max_{|s|\leq\widehat{\mathfrak{M}}_{\vec{\boldsymbol{\mathrm{r}}}}}|\mathcal{N}_j(s)|+|\mathcal{N}_j(\vec{\boldsymbol{\mathrm{r}}}\boldsymbol{\cdot}\overrightarrow{\boldsymbol{\mathrm{H}}})|\right)\left(\frac{|\partial\Omega|}{|\Omega|}\right)^{1-\frac1p}\left(||w_j||_{\text{L}^{p}(\Omega)}+\boldsymbol{\mathrm{o}}_{\lambda}\right)\\
&\,+\mathcal{L}_{\vec{\boldsymbol{\mathrm{r}}}}|\Omega|^{1-\frac1p}||w_j||_{\text{L}^{p}(\Omega)}\left(\left(kL||\vec{\boldsymbol{\mathrm{r}}}||_{\ell^\infty}\lambda^{\tau}+1\right)||\overrightarrow{\boldsymbol{\mathrm{H}}}||_{\mathrm{L}^{\infty}(\Omega)}+\boldsymbol{\mathrm{o}}_{\lambda}\right){\lambda^{\frac{1-p}{4p}}}\\
=&\,\left(\mathcal{L}_{\vec{\boldsymbol{\mathrm{r}}}}|\Omega|^{1-\frac1p}||w_j||_{\text{L}^{p}(\Omega)}||\overrightarrow{\boldsymbol{\mathrm{H}}}||_{\mathrm{L}^{\infty}(\Omega)}+\left(\max_{|s|\leq\widehat{\mathfrak{M}}_{\vec{\boldsymbol{\mathrm{r}}}}}|\mathcal{N}_j(s)|+||\vec{\boldsymbol{\mathrm{r}}}||_{\ell^\infty}+1\right)\boldsymbol{\mathrm{o}}_{\lambda}\right)\lambda^{\frac{1-p}{4p}}.   
\end{aligned}
\end{equation}
Here we have used the fact that $\tau<0$ and $\lambda^{\frac{1-p}{3p}}\ll\lambda^{\frac{1-p}{4p}}$ as $\lambda\to\infty$. Let us recall the condition \eqref{Lr}. Then, for each $\vec{\boldsymbol{\mathrm{r}}}\in\boldsymbol{\mathrm{S}}_k$, we ensure the existence of a positive constant $\widehat{\lambda}_{\vec{\boldsymbol{\mathrm{r}}},L,\tau}^*\geq\widehat{\lambda}_{\vec{\boldsymbol{\mathrm{r}}},L,\tau}$ such that for $\lambda\geq\widehat{\lambda}_{\vec{\boldsymbol{\mathrm{r}}},L,\tau}^*$, we have
\begin{equation*}
\left(\max_{|s|\leq\widehat{\mathfrak{M}}_{\vec{\boldsymbol{\mathrm{r}}}}}|\mathcal{N}_j(s)|+||\vec{\boldsymbol{\mathrm{r}}}||_{\ell^\infty}+1\right){\boldsymbol{\mathrm{o}}}_{\lambda}<\frac1{k}-\mathcal{L}_{\vec{\boldsymbol{\mathrm{r}}}}|\Omega|^{1-\frac1p}\max_{1\leq j\leq k}||w_j||_{\mathrm{L}^{p}(\Omega)}||\overrightarrow{\boldsymbol{\mathrm{H}}}||_{\mathrm{L}^{\infty}(\Omega)}.
\end{equation*}
This along with \eqref{ch2-T} yields the endomorphism
\begin{equation}\label{sTt}
\boldsymbol{\Gamma}_{\vec{\boldsymbol{\mathrm{r}}},\lambda,\eta}(\boldsymbol{\mathrm{Q}}_{\lambda})\subseteq\boldsymbol{\mathrm{Q}}_{\lambda}\quad\text{for}\,\,\lambda\geq\widehat{\lambda}_{\vec{\boldsymbol{\mathrm{r}}},L,\tau}^*.
\end{equation}

Regarding the estimate of \eqref{ch-T} corresponding to the Type II integral of \eqref{N-3j}, we can still obtain the same estimate as in \eqref{sTt} by following an argument similar to the one used for the Type I integral and an analysis analogous to \eqref{ch-19}--\eqref{ch-23}.

Based on the endomorphism of $\boldsymbol{\Gamma}_{\vec{\boldsymbol{\mathrm{r}}},\lambda,\eta}$ for $\lambda\geq\widehat{\lambda}_{\vec{\boldsymbol{\mathrm{r}}},L,\tau}^*$, it remains to show that for sufficiently large~$\lambda$, the mapping $\boldsymbol{\Gamma}_{\vec{\boldsymbol{\mathrm{r}}},\lambda,\eta}:{\boldsymbol{\mathrm{Q}}_{\lambda}}\to{\boldsymbol{\mathrm{Q}}_{\lambda}}$ is a contraction, which guarantees the existence of a unique fixed point on $\boldsymbol{\mathrm{Q}}_{\lambda}$. Since the estimate for the Type~II integral of \eqref{N-3j} can be handled in a similar fashion as for the Type~I integral, we will only provide a detailed analysis for the case of the Type~I integral. This analysis will determine the specific range of $\lambda$ for which the mapping becomes a contraction on $\boldsymbol{\mathrm{Q}}_{\lambda}$.

Accordingly, for $\vec{\boldsymbol{\vartheta}}_1,\vec{\boldsymbol{\vartheta}}_2\in\boldsymbol{\mathrm{Q}}_{\lambda}$, we will estimate $\|\boldsymbol{\Gamma}_{\vec{\boldsymbol{\mathrm{r}}},\lambda,\eta}(\vec{\boldsymbol{\vartheta}}_1)-\boldsymbol{\Gamma}_{\vec{\boldsymbol{\mathrm{r}}},\lambda,\eta}(\vec{\boldsymbol{\vartheta}_2})\|_{\ell^\infty}$ by focusing on the Type I integral of \eqref{N-2j}. As will be shown, one may check from \eqref{N-2j} and \eqref{nT}--\eqref{vP} that
\begin{equation}\label{0j}
    \begin{aligned}
&\|\boldsymbol{\Gamma}_{\vec{\boldsymbol{\mathrm{r}}},\lambda,\eta}(\vec{\boldsymbol{\vartheta}}_1)-\boldsymbol{\Gamma}_{\vec{\boldsymbol{\mathrm{r}}},\lambda,\eta}(\vec{\boldsymbol{\vartheta}_2})\|_{\ell^\infty}\\
\leq&\, \max_{1\leq j\leq k}\left|\boldsymbol{\mathsf{N}}_j\big[(\vec{\boldsymbol{\mathrm{r}}}+\vec{\boldsymbol{\vartheta}}_1)\boldsymbol{\cdot}\overrightarrow{\boldsymbol{\Phi}}_{\lambda,\eta}+\Psi_{0,\lambda}\big]-\boldsymbol{\mathsf{N}}_j\big[(\vec{\boldsymbol{\mathrm{r}}}+\vec{\boldsymbol{\vartheta}}_2)\boldsymbol{\cdot}\overrightarrow{\boldsymbol{\Phi}}_{\lambda,\eta}+\Psi_{0,\lambda}\big]\right|\\
\leq&\,\max_{1\leq j\leq k}\int_{\Omega}|w_j|\left|\mathcal{N}_j((\vec{\boldsymbol{\mathrm{r}}}+\vec{\boldsymbol{\vartheta}}_1)\boldsymbol{\cdot}\overrightarrow{\boldsymbol{\Phi}}_{\lambda,\eta}+\Psi_{0,\lambda})-\mathcal{N}_j((\vec{\boldsymbol{\mathrm{r}}}+\vec{\boldsymbol{\vartheta}}_2)\boldsymbol{\cdot}\overrightarrow{\boldsymbol{\Phi}}_{\lambda,\eta}+\Psi_{0,\lambda})\right|\dy.
    \end{aligned}
\end{equation}
For simplicity, we set
\begin{equation*}
\mathfrak{I}_j:=|w_j|\left|\mathcal{N}_j((\vec{\boldsymbol{\mathrm{r}}}+\vec{\boldsymbol{\vartheta}}_1)\boldsymbol{\cdot}\overrightarrow{\boldsymbol{\Phi}}_{\lambda,\eta}+\Psi_{0,\lambda})-\mathcal{N}_j((\vec{\boldsymbol{\mathrm{r}}}+\vec{\boldsymbol{\vartheta}}_2)\boldsymbol{\cdot}\overrightarrow{\boldsymbol{\Phi}}_{\lambda,\eta}+\Psi_{0,\lambda})\right|,~1\leq j\leq k.  
\end{equation*}
Recall that $\mathcal{N}_j:\mathbb{R}\to\mathbb{R}$ is locally Lipschitz continuous. Then by \eqref{le-*i}, \eqref{OMV} and \eqref{mPj}, for $\lambda\geq\widehat{\lambda}_{\vec{\boldsymbol{\mathrm{r}}},L,\tau}^*$ sufficiently large, there exists a constant  $\mathcal{L}_{\vec{\boldsymbol{\mathrm{r}}},j}^*>0$ independent of $\lambda$ such that
\begin{equation}\label{2j}
    \begin{aligned}
 \int_{\Omega_{\lambda^{-\frac13}}}\mathfrak{I}_j(y)\dy \leq&\, 
\mathcal{L}_{\vec{\boldsymbol{\mathrm{r}}},j}^*\int_{\Omega_{\lambda^{-\frac13}}}|w_j|\left|(\vec{\boldsymbol{\vartheta}}_1-\vec{\boldsymbol{\vartheta}}_2)\boldsymbol{\cdot}\overrightarrow{\boldsymbol{\Phi}}_{\lambda,\eta}\right|\dy\\
\leq&\,\frac{k\eta}{\lambda}\mathcal{L}_{\vec{\boldsymbol{\mathrm{r}}},j}^*||\overrightarrow{\boldsymbol{\mathrm{H}}}||_{\mathrm{L}^{\infty}(\Omega)}\|\vec{\boldsymbol{\vartheta}}_1-\vec{\boldsymbol{\vartheta}_2}\|_{\ell^\infty}\int_{\Omega_{\lambda^{-\frac13}}}|w_j|\dy\\
\leq&\,\frac{k\eta}{\lambda}\mathcal{L}_{\vec{\boldsymbol{\mathrm{r}}},j}^*||\overrightarrow{\boldsymbol{\mathrm{H}}}||_{\mathrm{L}^{\infty}(\Omega)}|\Omega_{\lambda^{-\frac13}}|^{1-\frac1p}||w_j||_{\text{L}^{p}(\Omega)}\|\vec{\boldsymbol{\vartheta}}_1-\vec{\boldsymbol{\vartheta}_2}\|_{\ell^\infty}\\
\leq&\,k\lambda^{\frac{1-p}{3p}}\left(1+L\lambda^{\tau+\frac{1-p}{4p}}\right)\left(|\partial\Omega|^{1-\frac1p}+\boldsymbol{\mathrm{o}}_{\lambda}\right)\mathcal{L}_{\vec{\boldsymbol{\mathrm{r}}},j}^*\\
&\,\times||\overrightarrow{\boldsymbol{\mathrm{H}}}||_{\mathrm{L}^{\infty}(\Omega)}||w_j||_{\text{L}^{p}(\Omega)}\|\vec{\boldsymbol{\vartheta}}_1-\vec{\boldsymbol{\vartheta}_2}\|_{\ell^\infty}.
    \end{aligned}
\end{equation}
Here, the last line of \eqref{2j} is verified using \eqref{le-*i} and \eqref{OMV}. Note also that
\begin{equation}\label{2j1}
k\lambda^{\frac{1-p}{3p}}\left(1+L\lambda^{\tau+\frac{1-p}{4p}}\right)\left(|\partial\Omega|^{1-\frac1p}+\boldsymbol{\mathrm{o}}_{\lambda}\right)\xrightarrow{\lambda\to\infty}0
\end{equation}
trivially due to $p>1$ and $\tau<0$.

Recall that for $\lambda\geq\widehat{\lambda}_{\vec{\boldsymbol{\mathrm{r}}},L,\tau}^*\geq\widehat{\lambda}_{\vec{\boldsymbol{\mathrm{r}}},L,\tau}$, \eqref{d-L2} is valid for all $\vec{\boldsymbol{\vartheta}}\in\boldsymbol{\mathrm{Q}}_{\lambda}$. Hence, for $\vec{\boldsymbol{\vartheta}}_1,\vec{\boldsymbol{\vartheta}}_2\in\boldsymbol{\mathrm{Q}}_{\lambda}$,  by \eqref{sk} and \eqref{d-L2} we have 
\begin{equation}\label{1j2}
    \begin{aligned}
\int_{\Omega\setminus\overline{\Omega_{\lambda^{-\frac13}}}}\mathfrak{I}_j(y)\dy\leq&\,\mathcal{L}_{\vec{\boldsymbol{\mathrm{r}}}}\int_{\Omega\setminus\overline{\Omega_{\lambda^{-\frac13}}}}|w_j|\left|(\vec{\boldsymbol{\vartheta}}_1-\vec{\boldsymbol{\vartheta}}_2)\boldsymbol{\cdot}\overrightarrow{\boldsymbol{\Phi}}_{\lambda,\eta}\right|\dy\\
\leq&\,k\mathcal{L}_{\vec{\boldsymbol{\mathrm{r}}}}|\Omega|^{1-\frac1p}||w_j||_{\text{L}^{p}(\Omega)}\|\vec{\boldsymbol{\vartheta}}_1-\vec{\boldsymbol{\vartheta}_2}\|_{\ell^\infty}\max_{1\leq j\leq k}\max_{\Omega\setminus\overline{\Omega_{\lambda^{-\frac13}}}}|\Phi_{j,\lambda,\eta}|.
    \end{aligned}    
\end{equation}
On the other hand, by \eqref{HN}, \eqref{le-*i} and \eqref{8-ch} we have
\begin{equation}\label{3j}
    \begin{aligned}  \max_{1\leq j\leq k}\max_{\Omega\setminus\overline{\Omega_{\lambda^{-\frac13}}}}|\Phi_{j,\lambda,\eta}|\leq&\,\max_{1\leq j\leq k}\max_{\Omega\setminus\overline{\Omega_{\lambda^{-\frac13}}}}\left|\Phi_{j,\lambda,\eta}-\frac{h_j}{f}\right|+||\overrightarrow{\boldsymbol{\mathrm{H}}}||_{\mathrm{L}^{\infty}(\Omega)}\\
    \leq&\,\left(\frac{|\eta-\lambda|}{\lambda}+1\right)||\overrightarrow{\boldsymbol{\mathrm{H}}}||_{\mathrm{L}^{\infty}(\Omega)}+\frac{\eta}{\lambda^2}\mathfrak{M}_{3,j}\\
&\,+\frac{\eta}{\lambda}\left(2||\overrightarrow{\boldsymbol{\mathrm{H}}}||_{\mathrm{L}^{\infty}(\Omega)}+ \frac{\mathfrak{M}_{3,j}}{\lambda}\right)\exp\left({-\frac{\mathfrak{M}_0}2\lambda^{\frac16}}\right)\\
\xrightarrow{\lambda\to\infty}&\,||\overrightarrow{\boldsymbol{\mathrm{H}}}||_{\mathrm{L}^{\infty}(\Omega)}.
    \end{aligned}
\end{equation}
Here we have used \eqref{le-*i} to verify $\frac{|\eta-\lambda|}{\lambda}\leq L\lambda^{\tau+\frac{1-p}{4p}}\xrightarrow{\lambda\to\infty}0$ and $\frac\eta{\lambda^2}\xrightarrow{\lambda\to\infty}0$.

Since $k\mathcal{L}_{\vec{\boldsymbol{\mathrm{r}}}}|\Omega|^{1-\frac1p}||w_j||_{\text{L}^{p}(\Omega)}||\overrightarrow{\boldsymbol{\mathrm{H}}}||_{\mathrm{L}^{\infty}(\Omega)}<1$ (cf. \eqref{Lr}), we combine the estimates from \eqref{2j}--\eqref{2j1} and \eqref{1j2}--\eqref{3j} to show that there exists a constant $\lambda_{\vec{\boldsymbol{\mathrm{r}}},L,\tau}^*\geq\widehat{\lambda}_{\vec{\boldsymbol{\mathrm{r}}},L,\tau}^*$ such that for all $\lambda\geq\lambda_{\vec{\boldsymbol{\mathrm{r}}},L,\tau}^*$, we have
\begin{align*}
  \max_{1\leq j\leq k} \int_{\Omega}\mathfrak{I}_j(y)\dy \leq&\,\frac12\left(1+k\mathcal{L}_{\vec{\boldsymbol{\mathrm{r}}}}|\Omega|^{1-\frac1p}||w_j||_{\text{L}^{p}(\Omega)}||\overrightarrow{\boldsymbol{\mathrm{H}}}||_{\mathrm{L}^{\infty}(\Omega)}\right) \|\vec{\boldsymbol{\vartheta}}_1-\vec{\boldsymbol{\vartheta}_2}\|_{\ell^\infty}.
\end{align*}
This result, together with \eqref{0j}, shows that the mapping $\boldsymbol{\Gamma}_{\vec{\boldsymbol{\mathrm{r}}},\lambda,\eta}:{\boldsymbol{\mathrm{Q}}_{\lambda}}\to{\boldsymbol{\mathrm{Q}}_{\lambda}}$ is a contraction and thus admits a unique fixed point $\vec{\boldsymbol{\vartheta}}_{\vec{\boldsymbol{\mathrm{r}}},\lambda,\eta}$ for all $\lambda\geq\lambda_{\vec{\boldsymbol{\mathrm{r}}},L,\tau}^*$. This completes the proof of Proposition~\ref{prop-3}.
\end{proof}

 Having Proposition~\ref{prop-3} at hand, we now present the proof of Theorem~\ref{thm2}. 
 \begin{proof}[Proof of Theorem~\ref{thm2}]
 We first consider the case where $\boldsymbol{\mathrm{S}}_k=\{\vec{\boldsymbol{\mathrm{r}}}_1,...,\vec{\boldsymbol{\mathrm{r}}}_n\}$ is a finite set of $n$ distinct elements. By Proposition~\ref{prop-3}, for $\lambda\geq\max_{\color{black} 1\leq j\leq n}\lambda_{\vec{\boldsymbol{\mathrm{r}}}_j,L,\tau}^*$, and provided that $\eta$ satisfies \eqref{le-*i}, each mapping $\boldsymbol{\Gamma}_{\vec{\boldsymbol{\mathrm{r}}}_j,\lambda,\eta}:{\boldsymbol{\mathrm{Q}}_{\lambda}}\to{\boldsymbol{\mathrm{Q}}_{\lambda}}$ defined in \eqref{nT} admits a unique fixed point $\vec{\boldsymbol{\vartheta}}_{\vec{\boldsymbol{\mathrm{r}}}_j,\lambda,\eta}$ that satisfies $\|\vec{\boldsymbol{\vartheta}}_{\vec{\boldsymbol{\mathrm{r}}}_j,\lambda,\eta}\|_{\ell^\infty}\leq{k^{-1}\lambda^{\frac{1-p}{4p}}}$. This implies that,  for all $j=1,...,k$,  $u_{\vec{\boldsymbol{\mathrm{r}}}_j,\lambda,\eta}$ as defined in \eqref{rtu} with $\vec{\boldsymbol{\mathrm{r}}}=\vec{\boldsymbol{\mathrm{r}}}_j$ is a solution to equation~\eqref{eq-u2} with the boundary condition~\eqref{rbd}. By \eqref{rtu}, \eqref{le-*i} and \eqref{mPj}, for $\lambda\geq\max_{\color{black} 1\leq j\leq n}\lambda_{\vec{\boldsymbol{\mathrm{r}}}_j,L,\tau}^*$, we arrive at
\begin{equation}\label{r-ut}
\begin{aligned}
\max_{\overline{\Omega}}\left|u_{\vec{\boldsymbol{\mathrm{r}}}_j,\lambda,\eta}-\left(\vec{\boldsymbol{\mathrm{r}}}_j\boldsymbol{\cdot}\overrightarrow{\boldsymbol{\Phi}}_{\lambda,\eta}+\Psi_{0,\lambda}\right)\right|\leq&\,\frac{k\eta}{\lambda}\|\vec{\boldsymbol{\vartheta}}_{\vec{\boldsymbol{\mathrm{r}}}_j,\lambda,\eta}\|_{\ell^\infty}\|\overrightarrow{\boldsymbol{\mathrm{H}}}\|_{\mathrm{L}^{\infty}(\Omega)}\\
\leq&\,\left(1+L\lambda^{\tau+\frac{1-p}{4p}}\right)\lambda^{\tau+\frac{1-p}{4p}}\|\overrightarrow{\boldsymbol{\mathrm{H}}}\|_{\mathrm{L}^{\infty}(\Omega)}.
\end{aligned}
\end{equation}
Additionally, since $\vec{\boldsymbol{\vartheta}}_{\vec{\boldsymbol{\mathrm{r}}}_j,\lambda,\eta}$ is a fixed point of the mapping in \eqref{nT} with $\vec{\boldsymbol{\mathrm{r}}}=\vec{\boldsymbol{\mathrm{r}}}_j$, it follows that
\begin{equation}\label{theta*}
\left|\big\langle\boldsymbol{\mathsf{N}}_1\big[u_{\vec{\boldsymbol{\mathrm{r}}}_j,\lambda,\eta}\big],...,\boldsymbol{\mathsf{N}}_k\big[u_{\vec{\boldsymbol{\mathrm{r}}}_j,\lambda,\eta}\big] \big\rangle-\vec{\boldsymbol{\mathrm{r}}}_j\right|\leq\|\vec{\boldsymbol{\vartheta}}_{\vec{\boldsymbol{\mathrm{r}}}_j,\lambda,\eta}\|_{\ell^\infty}\leq{k^{-1}\lambda^{\frac{1-p}{4p}}},~\text{for}~{\color{black}1\leq j\leq n}.
\end{equation}
Consequently, \eqref{u-1sa} is a direct result of \eqref{r-ut} and \eqref{theta*}.

We now show that there exists a constant $\lambda^*\geq\max_{\color{black} 1\leq j\leq n}\lambda_{\vec{\boldsymbol{\mathrm{r}}}_j,L,\tau}^*$ such that for all $\lambda\geq\lambda^*$ and $\eta$ satisfying \eqref{le-*i}, all solutions $u_{\vec{\boldsymbol{\mathrm{r}}}_1,\lambda,\eta},...,u_{\vec{\boldsymbol{\mathrm{r}}}_n,\lambda,\eta}$ are distinct.
We proceed by contradiction as follows. 

Suppose that there exist $\vec{\boldsymbol{\mathrm{r}}}_{j_1}\neq\vec{\boldsymbol{\mathrm{r}}}_{j_2}$ such that
\begin{equation}\label{u*}
u_{\vec{\boldsymbol{\mathrm{r}}}_{j_1},\widetilde{\lambda}_i,\widetilde{\eta}_i}=u_{\vec{\boldsymbol{\mathrm{r}}}_{j_2},\widetilde{\lambda}_i,\widetilde{\eta}_i}
\end{equation}
for a sequence $\{(\widetilde{\lambda}_i,\widetilde{\eta}_i)\}_{i\in\mathbb{N}}$ where $\lim_{i\to\infty}\widetilde{\lambda}_i=\infty$ and \eqref{le-*i} holds with $(\lambda,\eta)=(\widetilde{\lambda}_i,\widetilde{\eta}_i)$. From \eqref{r-ut} and \eqref{u*}, it follows that 
\begin{equation}\label{*u}
\begin{aligned}
&\max_{\overline{\Omega}}\left|(\vec{\boldsymbol{\mathrm{r}}}_{j_1}-\vec{\boldsymbol{\mathrm{r}}}_{j_2})\boldsymbol{\cdot}\overrightarrow{\boldsymbol{\Phi}}_{\widetilde{\lambda}_i,\widetilde{\eta}_i}\right|\\
\leq&\,\max_{\overline{\Omega}}\left(\left|u_{\vec{\boldsymbol{\mathrm{r}}}_{\color{black}j_1},\widetilde{\lambda}_i,\widetilde{\eta}_i}-(\vec{\boldsymbol{\mathrm{r}}}_{\color{black}j_1}\boldsymbol{\cdot}\overrightarrow{\boldsymbol{\Phi}}_{\lambda,\eta}+\Psi_{0,\widetilde{\lambda}_i})\right|+\left|u_{\vec{\boldsymbol{\mathrm{r}}}_{\color{black}j_2},\lambda,\eta}-(\vec{\boldsymbol{\mathrm{r}}}_{j_2}\boldsymbol{\cdot}\overrightarrow{\boldsymbol{\Phi}}_{\widetilde{\lambda}_i,\widetilde{\eta}_i}+\Psi_{0,\widetilde{\lambda}_i})\right|\right) \\
\leq&\,2\left(1+L\widetilde{\lambda}_i^{\tau+\frac{1-p}{4p}}\right)\widetilde{\lambda}_i^{\tau+\frac{1-p}{4p}}\|\overrightarrow{\boldsymbol{\mathrm{H}}}\|_{\mathrm{L}^{\infty}(\Omega)}\to0\quad\text{as}\,\,i\to\infty.
\end{aligned}
\end{equation}
We recall from (A4) that the set $\{h_1(\boldsymbol{\mathrm{x}}_0),...,h_k(\boldsymbol{\mathrm{x}}_0)\}$ is linearly independent for an interior point $\boldsymbol{\mathrm{x}}_0\in\Omega$. Note also that the distance $\mathsf{d}(\boldsymbol{\mathrm{x}}_0)>0$ is independent of $\widetilde{\lambda}_i$. Thus, by virtue of \eqref{HN}, \eqref{vP} and \eqref{8-ch}, we immediately obtain
\begin{equation*}
\lim_{i\to\infty}\overrightarrow{\boldsymbol{\Phi}}_{\widetilde{\lambda}_i,\widetilde{\eta}_i}(\boldsymbol{\mathrm{x}}_0)=\overrightarrow{\boldsymbol{\mathrm{H}}}(\boldsymbol{\mathrm{x}}_0)=\big\langle \frac{h_1(\boldsymbol{\mathrm{x}}_0)}{f(\boldsymbol{\mathrm{x}}_0)},...,\frac{h_k(\boldsymbol{\mathrm{x}}_0)}{f(\boldsymbol{\mathrm{x}}_0)} \big\rangle.
\end{equation*}
This, along with \eqref{*u}, leads to $(\vec{\boldsymbol{\mathrm{r}}}_{j_1}-\vec{\boldsymbol{\mathrm{r}}}_{j_2})\boldsymbol{\cdot}\big\langle h_1(\boldsymbol{\mathrm{x}}_0),...,h_k(\boldsymbol{\mathrm{x}}_0) \big\rangle=0$, which by the linear independence of $\{h_1(\boldsymbol{\mathrm{x}}_0),...,h_k(\boldsymbol{\mathrm{x}}_0)\}$ implies $\vec{\boldsymbol{\mathrm{r}}}_{j_1}=\vec{\boldsymbol{\mathrm{r}}}_{j_2}$. This contradiction invalidates our initial assumption in \eqref{u*} and, therefore, completes the proof of Theorem~\ref{thm2}(i).

Theorem~\ref{thm2}(ii) is a direct consequence of Theorem~\ref{thm2}(i). By considering a subset $\{\vec{\boldsymbol{\mathrm{r}}}_1,...,\vec{\boldsymbol{\mathrm{r}}}_{\widetilde{n}}\}$ of $\boldsymbol{\mathrm{S}}_k$  and following the same argument as in the proof of Theorem~\ref{thm2}(i), we obtain that there exists $\lambda_{\widetilde{n}}^*>0$ such that for all $\lambda\geq\lambda_{\widetilde{n}}^*$, the equation~\eqref{eq-u2} with the boundary condition~\eqref{rbd} possesses at least $\widetilde{n}$ distinct solutions. This completes the proof of Theorem~\ref{thm2}.
 \end{proof}

%\section{\bf Concluding remarks}

\section{\bf Concluding remarks and outlook}\label{sec:conclusion}
In this work, we study the qualitative properties of linear elliptic equations with multiple nonlocal nonlinearities. These equations present significant challenges, as they are typically not amenable to standard variational methods and often fall outside the scope of classical elliptic theory. To address the difficulty in analysis, we developed a novel theoretical framework that integrates fixed-point arguments with asymptotic techniques. This approach establishes a new connection between asymptotic analysis and the existence and multiplicity theory of non-variational elliptic equations. Our methodology not only circumvents the absence of a variational structure but also uncovers a deeper understanding of how the interplay between parameters and nonlocal effects dictates the number of solutions. 

Beyond a purely mathematical study, our work provides a new perspective on the fundamental nature of these nonlocal equations. The diverse asymptotics of solutions we have uncovered suggests that the interplay between parameters and nonlocal effects dictates not just existence, but also the very structure of these solutions. This finding is significant as it provides a new paradigm for understanding how nonlocality fundamentally alters the behavior of nonlocal elliptic equations, with potential implications for modeling complex phenomena in physics and biology.

 Our findings contribute to a broader effort to understand the intricate interplay between locality and nonlocality in mathematical models of the physical systems. Furthermore, our results motivate a deeper exploration into the specific physical interpretations of the multiple solutions we have found. For instance, in applications like population dynamics or neural networks, these distinct solutions could correspond to different stable states or equilibrium patterns. We believe this work provides a solid foundation for future research and underscores the importance of developing new analytical tools for nonlocal problems. In the coming years, we expect this framework to be extended to investigate equations with more complex nonlocal terms or different boundary conditions.

\section{\bf Appendix: Validity of assumptions~(A3) and (A6)}\label{BPX}
We will provide examples on the validity of assumptions (A3) and (A6), which are constructed by means of periodically perturbed functions. 

{\bf Validity of assumption~(A3).}~To demonstrate the validity of assumption (A3), we can consider a simplified case where $w(y)=\frac1{|\Omega|}$. This choice satisfies assumption (A2) and implies that the condition on the Lipschitz constant in \eqref{s0n} becomes $\mathcal{L}_r<\frac{|\Omega|^{\frac1p-1}}{||w||_{\mathrm{L}^{p}(\Omega)}}=1$. We now provide an example for \eqref{s0n} satisfying (A3) as follows.
\begin{itemize}
    \item[\bf(B1)]  Consider a cubic polynomial
    \begin{equation*}
 \mathcal{N}(r)=r-\frac2{(r_1-r_2)^2}(r-r_1)(r-r_2)\left(r-\frac{r_1+r_2}2\right),       
    \end{equation*}
where $r_1\neq r_2$ are reals. Then the equation $r=\mathcal{N}(r)$ has three distinct real roots $r_1$, $r_2$ and $\frac{r_1+r_2}2$ satisfying $\mathcal{N}'(r_1)=0$, $\mathcal{N}'(r_2)=0$ and $\mathcal{N}'(\frac{r_1+r_2}2)=\frac32$. It immediately follows $\boldsymbol{\mathrm{S}}_{\star}=\{r_1,r_2\}$ for this specific case. 
\end{itemize}

In fact, high-degree polynomials $\mathcal{N}(r)$ that satisfy condition (A3) are frequently encountered. 

For the reader's reference, we provide examples~(B2)--(B4) constructed from periodically perturbed functions to illustrate a case where the set $\boldsymbol{\mathrm{S}}_{\star}$ contains infinitely many elements.  
\begin{itemize}
\item[\bf(B2)] For $\mathcal{N}(r)=r+\sin r$, the roots of the equation~$r=\mathcal{N}(r)$ are given by $j\pi$ for $j\in\mathbb{Z}$.
Although this $\mathcal{N}:\mathbb{R}\to\mathbb{R}$  is globally Lipschitz continuous with an optimal constant of $2$, a much tighter local Lipschitz condition holds in certain neighborhoods of the roots $r_j=(2j-1)\pi$:\footnote{Note that $\mathcal{N}'((2j-1)\pi)=0$. This estimate is only used to align with the definition of the set $\boldsymbol{\mathrm{S}}_{\star}$ in \eqref{s0n}.}
 \begin{equation*}
 |\mathcal{N}(s_1)-\mathcal{N}(s_2)|\leq \frac12|s_1-s_2|~\text{for}~s_1,s_2\in(r_j-\frac\pi3,r_j+\frac\pi3),~\forall i\in\mathbb{Z},
 \end{equation*}
which implies $ \boldsymbol{\mathrm{S}}_{\star}=\{(2j-1)\pi|\,j\in\mathbb{Z}\}$.
 This is an important example showing that even if $\mathcal{N}:\mathbb{R}\to\mathbb{R}$ is globally Lipschitz continuous with  Lipschitz constant large than $\frac{|\Omega|^{\frac1p-1}}{||w||_{\mathrm{L}^{p}(\Omega)}}$, its corresponding set $\boldsymbol{\mathrm{S}}_{\star}$ can still have infinitely many elements.
    \item[\bf(B3)] Consider
\begin{equation*}
 \mathcal{N}(r)=r+\alpha\text{e}^{-r^2}\sin r\quad\text{with}\,\,\alpha\in(0,1],   
\end{equation*}
 for which the corresponding $\boldsymbol{\mathrm{S}}_{\star}=\{(2j-1)\pi|\,j\in\mathbb{Z}\}$. 
     \item[\bf(B4)] Consider $\mathcal{N}(r)=r+\xi(r)$, where
 \begin{equation*}
  \xi(r)=
  \begin{cases}
 \frac12r(1-r),\quad\,r\in [0,1),\\
 \frac12(r-1)(r-2),\quad\,r\in[1,2),
  \end{cases}
 \end{equation*}
 and extended periodically with a period of $2$. In this case,  $\boldsymbol{\mathrm{S}}_{\star}=\{2j-1|\,j\in\mathbb{Z}\}$.
\end{itemize}
 (B3) and (B4) can be straightforward to verify.

{\bf Validity of assumption~(A6).}~We provide examples to demonstrate the validity of assumption (A6). For simplicity, we set $k=2$ and let the functions $w_1$ and $w_2$ be defined on two disjoint subsets $\Omega_1,\Omega_2\subset\Omega$ with $|\Omega_1|=|\Omega_2|=\frac12|\Omega|$, respectively, as follows: 
\begin{equation}\label{ww}
w_1(y)=\frac2{|\Omega|}\chi_{\Omega_1},\quad\, w_2(y)=\frac2{|\Omega|}\chi_{\Omega_2}, 
\end{equation}
where $\chi_{\Omega_j}$ denotes the characteristic function on the set $\Omega_j$, $j=1,2$. \eqref{ww} immediately implies that $w_1$ and $w_2$ are linearly independent and satisfy $\int_{\Omega}w_1(y)\dy=\int_{\Omega}w_2(y)\dy=1$ and $||w_1||_{\mathrm{L}^{p}(\Omega)}=||w_2||_{\mathrm{L}^{p}(\Omega)}=(\frac{|\Omega|}2)^{\frac1p-1}$.

Next, we define the functions $h_j$'s and $f$ as  
\begin{equation}\label{hjf}
   \begin{cases}
h_1(x)=f(x)+\gamma_1,\quad~h_2(x)=2f(x)+\gamma_2,\\
\text{where}\,\,f\not\equiv\text{constant}\,\,\text{on}\,\,\overline{\Omega},\,\,\text{and}\,\,\gamma_j\text{'s~are~constants~satisfying}\,\,\gamma_2\neq2\gamma_1.   
   \end{cases} 
\end{equation}
This ensures that $h_1$ and $h_2$ are linearly independent on the domain $\overline{\Omega}$. On the other hand, the condition on all Lipschitz constants~$\mathcal{L}_{\vec{\boldsymbol{\mathrm{r}}}}$ in \eqref{Lr} reduces to 
\begin{equation}\label{LH}
0<\mathcal{L}_{\vec{\boldsymbol{\mathrm{r}}}}<\frac{2^{\frac1p-3}}{||\overrightarrow{\boldsymbol{\mathrm{H}}}||_{\mathrm{L}^{\infty}(\Omega)}}\,\,\text{for~all}\,\, \vec{\boldsymbol{\mathrm{r}}}\in\boldsymbol{\mathrm{S}}_2. 
\end{equation}

Since by \eqref{HN} and \eqref{hjf}, we have $\overrightarrow{\boldsymbol{\mathrm{H}}}(x)-\langle 1,2\rangle=\langle \frac{\gamma_1}{f(x)},\frac{\gamma_2}{f(x)}\rangle$ on $\overline{\Omega}$,
we will construct $\mathcal{N}_1$ and $\mathcal{N}_2$ and determine the values of $\gamma_1$ and $\gamma_2$ such that any solution $\vec{\boldsymbol{\mathrm{r}}}\in\mathbb{R}^2$ to the system $\vec{\boldsymbol{\mathrm{r}}}=\big\langle{\mathcal{N}}_1\big(\vec{\boldsymbol{\mathrm{r}}}\boldsymbol{\cdot}\langle1,2\rangle\big),{\mathcal{N}}_2\big(\vec{\boldsymbol{\mathrm{r}}}\boldsymbol{\cdot}\langle1,2\rangle\big) \big\rangle$ is also a solution to the system $\vec{\boldsymbol{\mathrm{r}}}=\big\langle\boldsymbol{\mathsf{N}}_1\big[\vec{\boldsymbol{\mathrm{r}}}\boldsymbol{\cdot}\overrightarrow{\boldsymbol{\mathrm{H}}}\big],\boldsymbol{\mathsf{N}}_2\big[\vec{\boldsymbol{\mathrm{r}}}\boldsymbol{\cdot}\overrightarrow{\boldsymbol{\mathrm{H}}}\big] \big\rangle$, i.e.,
\begin{equation}\label{n1sn2}
\begin{aligned}
  \small{  \{\vec{\boldsymbol{\mathrm{r}}}\in\mathbb{R}^2|\,\vec{\boldsymbol{\mathrm{r}}}=\big\langle{\mathcal{N}}_1\big(\vec{\boldsymbol{\mathrm{r}}}\boldsymbol{\cdot}\langle1,2\rangle\big),{\mathcal{N}}_2\big(\vec{\boldsymbol{\mathrm{r}}}\boldsymbol{\cdot}\langle1,2\rangle\big) \big\rangle\}\subseteq\{\vec{\boldsymbol{\mathrm{r}}}\in\mathbb{R}^2|\,\vec{\boldsymbol{\mathrm{r}}}=\big\langle\boldsymbol{\mathsf{N}}_1\big[\vec{\boldsymbol{\mathrm{r}}}\boldsymbol{\cdot}\overrightarrow{\boldsymbol{\mathrm{H}}}\big],\boldsymbol{\mathsf{N}}_2\big[\vec{\boldsymbol{\mathrm{r}}}\boldsymbol{\cdot}\overrightarrow{\boldsymbol{\mathrm{H}}}\big] \big\rangle\}. }   
\end{aligned}
\end{equation}
This enables the construction of the following examples~(B5)--(B6) satisfying assumption (A6), which are a bit nontrivial.
\begin{itemize}
    \item[\bf(B5)]  We consider 
 \begin{equation}\label{12N}
    \mathcal{N}_1(s)=-\frac{s}{2^{\rho}},\,\,\quad \mathcal{N}_2(s)=\frac{1+2^{\rho}}{2^{\rho+1}}\left(s-\frac{1+2^{\rho-1}}{1+2^{\rho}}(s-1)(s-2)(s-3)\right).
 \end{equation} 
We will show that for
\begin{equation}\label{rho6}
 \rho \geq6-\frac1p,  
\end{equation}
 the setting of functions in \eqref{12N} satisfies assumption~(A6).

 Solving the system $\vec{\boldsymbol{\mathrm{r}}}=\big\langle{\mathcal{N}}_1\big(\vec{\boldsymbol{\mathrm{r}}}\boldsymbol{\cdot}\langle1,2\rangle\big),{\mathcal{N}}_2\big(\vec{\boldsymbol{\mathrm{r}}}\boldsymbol{\cdot}\langle1,2\rangle\big) \big\rangle$ directly yields the set of solutions $\{\vec{\boldsymbol{\mathrm{r}}}_1=\langle-\frac{1}{2^{\rho}},\frac{1+2^{\rho}}{2^{\rho+1}}\rangle,\vec{\boldsymbol{\mathrm{r}}}_2=2\vec{\boldsymbol{\mathrm{r}}}_1,\vec{\boldsymbol{\mathrm{r}}}_3=3\vec{\boldsymbol{\mathrm{r}}}_1\}$. It is clear that when
 \begin{equation}\label{rho-ga}
\gamma_1=\frac{1+2^{\rho}}2\gamma_2,
 \end{equation}
there holds $\vec{\boldsymbol{\mathrm{r}}}_j\boldsymbol{\cdot}\overrightarrow{\boldsymbol{\mathrm{H}}}=\vec{\boldsymbol{\mathrm{r}}}_j\boldsymbol{\cdot}\langle1,2\rangle$, for $j=1,2,3$. Then by \eqref{N-2j}, this implies that $\vec{\boldsymbol{\mathrm{r}}}_j$ satisfy the system~$\vec{\boldsymbol{\mathrm{r}}}_j=\big\langle\boldsymbol{\mathsf{N}}_1\big[\vec{\boldsymbol{\mathrm{r}}}_j\boldsymbol{\cdot}\overrightarrow{\boldsymbol{\mathrm{H}}}\big],\boldsymbol{\mathsf{N}}_2\big[\vec{\boldsymbol{\mathrm{r}}}_j\boldsymbol{\cdot}\overrightarrow{\boldsymbol{\mathrm{H}}}\big] \big\rangle$ for $j=1,2,3$. Hence, we have checked that \eqref{n1sn2} holds.  A simple calculation yields 
\begin{equation*}
\begin{cases}
 \displaystyle  |{\mathcal{N}}_1'|\equiv\frac1{2^{\rho}},\,\,
|{\mathcal{N}}_2'\big(\vec{\boldsymbol{\mathrm{r}}}_1\boldsymbol{\cdot}\langle1,2\rangle\big)| =|{\mathcal{N}}_2'\big(\vec{\boldsymbol{\mathrm{r}}}_3\boldsymbol{\cdot}\langle1,2\rangle\big)|=\frac1{2^{\rho+1}},\\ 
\displaystyle{\mathcal{N}}_2'\big(\vec{\boldsymbol{\mathrm{r}}}_2\boldsymbol{\cdot}\langle1,2\rangle\big)=\frac{1+3\cdot2^{\rho-1}}{2^{\rho+1}}.
\end{cases}
\end{equation*}
By \eqref{LH} and \eqref{rho-ga}, we  set $\gamma_1$ and $\gamma_2$, and require the property of $f$ as follows:
\begin{equation}\label{b-g}
\begin{cases}
 \displaystyle \gamma_2=-1,~\gamma_1=-\frac{1+2^{\rho}}2~\text{and}~f~\text{satisfies}\\
  \min\limits_{\overline{\Omega}}f=2^{\rho-3}~\&~\max\limits_{\overline{\Omega}}f=2^{\rho-3}+1.   
\end{cases}
\end{equation}
A direct calculation yields $||\overrightarrow{\boldsymbol{\mathrm{H}}}||_{\mathrm{L}^{\infty}(\Omega)}=\frac{3\cdot2^{\rho}+12}{2^{\rho}-8}$. 

Under \eqref{rho6}, we claim that $\boldsymbol{\mathrm{S}}_2\supseteq\{\vec{\boldsymbol{\mathrm{r}}}_1,\vec{\boldsymbol{\mathrm{r}}}_3\}$ and $\mathcal{L}_{\vec{\boldsymbol{\mathrm{r}}}_j}=2^{-\rho}$ for $j=1,3$ satisfy \eqref{LH}.   This can be rigorously demonstrated by noting that $||\overrightarrow{\boldsymbol{\mathrm{H}}}||_{\mathrm{L}^{\infty}(\Omega)}<8$ for $\rho\geq4$, which further implies
$${2^{\frac1p-3}}{||\overrightarrow{\boldsymbol{\mathrm{H}}}||_{\mathrm{L}^{\infty}(\Omega)}^{-1}}-2^{-\rho}>2^{\frac1p-6}-2^{-\rho}\geq0~\text{for}~\rho\geq6-\frac1p.$$ 
Consequently,  $\boldsymbol{\mathrm{S}}_2\supseteq\{\vec{\boldsymbol{\mathrm{r}}}_1,\vec{\boldsymbol{\mathrm{r}}}_3\}$ has at least two distinct elements that satisfy assumption~(A6).
\item[\bf(B6)] As another example, let $\mathcal{N}_1$ be as defined in \eqref{12N}. We replace $\mathcal{N}_2$ from \eqref{12N}  with a periodically perturbed function
\begin{equation*}
  \mathcal{N}_2(s)=\frac{1+2^{\rho}}{2^{\rho+1}}\left(s-\frac{1-2^{\rho}}{1+2^{\rho}}\mathsf{k}\sin\frac{s}{\mathsf{k}}\right), 
\end{equation*}
 where $\mathsf{k}$ is a nonzero constant. Moreover, under assumptions \eqref{rho6}--\eqref{rho-ga} and \eqref{b-g}, it is straightforward to solve the system $\vec{\boldsymbol{\mathrm{r}}}=\big\langle{\mathcal{N}}_1\big(\vec{\boldsymbol{\mathrm{r}}}\boldsymbol{\cdot}\langle1,2\rangle\big),{\mathcal{N}}_2\big(\vec{\boldsymbol{\mathrm{r}}}\boldsymbol{\cdot}\langle1,2\rangle\big) \big\rangle$ and show that the corresponding set 
 $$\boldsymbol{\mathrm{S}}_2\supseteq\{\left\langle-\frac{2\mathsf{z}-1}{2^\rho}\mathsf{k}\pi,\frac{(1+2^\rho)(2\mathsf{z}-1)}{2^{\rho+1}}\mathsf{k}\pi\right\rangle|\,\mathsf{z}\in\mathbb{Z}\}$$
 contains infinitely many elements.
\end{itemize}

\subsection*{Acknowledgement}
This work was supported by the National Science and Technology Council of Taiwan under Grant 114-2115-M-007-013-MY2. The author expresses sincere gratitude to the anonymous referees for the insightful comments that greatly improved the original manuscript.


\begin{thebibliography}{99}
%\footnotesize
\bibitem{AB1996}
{\sc W. Allegretto and A. Barabanova}, 
{\em Positivity of solutions of elliptic equations with nonlocal terms}, 
Proc. Roy. Soc. Edinburgh Sect. A \textbf{126} (1996) 643--663.



 \bibitem{AR1981} {\sc U. Ascher, R.D. Russell}, {\em Reformulation of boundary value problems in ``standard'' form}, SIAM Rev. {\bf 23} (1981) 238--254.
\bibitem{B1987}{\sc H. Bellout}, 
{\em Blow-up of solutions of parabolic equations with nonlinear memory}, 
J. Differential Equations \textbf{70} (1987) 42--68.
\bibitem{BTV1967}
{\sc M. Bertero, G. Talenti and G. A. Viano}, {\em Scattering and bound state solutions for a class of nonlocal potentials (s-wave)}, Comm. Math. Phys. \textbf{6} (1967) 128--150.

\bibitem{BDS1993}
{\sc C. Budd, B. Dold and A. Stuart}, 
{\em Blowup in a partial differential equation with conserved first integral}, 
SIAM J. Appl. Math. \textbf{53} (1993) 718--742.
\bibitem{CR2015}{\sc T. Cecil,  P.J. Ryan}, {\em Geometry of hypersurfaces.} {\bf Springer Monographs in Mathematics.} Springer, New York. xi+596 pp.  (2015).


\bibitem{CLM2024}{\sc X. Chen, C.-C. Lee, M. Mizuno}, {\em Unified asymptotic analysis and numerical simulations of singularly perturbed linear differential equations under various nonlocal boundary effects} (2024), Commun. Math. Sci. {\bf 22} (2024), no. 2, 395--434.

\bibitem{CLW2025}{\sc X. Chen, C.-C. Lee, W. Yang},
{\em Asymptotics and computation for a class of Fredholm integro-differential equations}, Discrete Contin. Dyn. Syst.~{\bf 45} (2025) 1008--1044.

\bibitem{Evans2015}
{\sc L.C. Evans, R. F. Gariepy},
{\em Measure Theory and Fine Properties of Functions},
2nd revised ed., Textbooks in Mathematics, CRC Press, 2015.


\bibitem{FS1998}
{\sc P. Freitas, G. Sweers},
{\rm Positivity results for a nonlocal elliptic equation},
Proc. Roy. Soc. Edinburgh Sect. A: Math. {\bf 128} (1998) 697--715.

\bibitem{FY2015}{\sc G. Freiling and V. Yurko}, {\em Recovering non-selfadjoint differential pencils with nonseparated boundary conditions}, Appl. Anal. {\bf 94}  (2015) 1649--1661.

%\bibitem{FD2017}{\sc A. Foroush Bastani,  D. Damircheli},   {\em An adaptive algorithm for solving stochastic multi-point boundary value problems},  Numer. Algorithms {\bf  74} (2017) 1119--1143. 

%\bibitem{G2012}{\sc F.Z. Geng},  {\em A numerical algorithm for nonlinear multi-point boundary value problems}, J. Comput. Appl. Math. {\bf 236} (2012) 1789--1794.


\bibitem{GT1983} {\sc D. Gilbarg, N.S. Trudinger}, 
{\em Elliptic partial differential equations of second order}, 
Springer-Verlag, New York, Heidelberg, and Berlin (1983).

\bibitem{HL1997} {\sc Q. Han,  F.-H. Lin},  {\em Elliptic Partial Differential Equations, Courant Lecture
Notes, v.1, Amer. Math. Soc.} (1997)

\bibitem{Lee2016}{\sc C.-C. Lee}, {\em Asymptotic analysis of charge conserving Poisson–Boltzmann equations with variable dielectric coefficients}, Discrete Contin. Dyn. Syst. {\bf 36} (2016) 3251--3276. 

\bibitem{Lee2019}{\sc C.-C. Lee}, {\em Thin layer analysis of a non-local model for the double layer structure},  J. Differ. Equ. {\bf 266} (2019) 742--802. 

%\bibitem{Lee2020}{\sc C.-C. Lee}, {\em Nontrivial boundary structure in a Neumann problem on balls with radii tending to infinity}, Ann. Mat. Pura Appl. {\bf 199} (2020), no. 3, 1123--1146.

\bibitem{L2023}{\sc C.-C. Lee}, {\em Uniqueness and asymptotics of singularly perturbed equations involving implicit boundary conditions}, Rev. R. Acad. Cienc. Exactas Fís. Nat. Ser. A Mat. RACSAM {\bf 117} (2023), Paper No. 51, 18 pages.

\bibitem{L2024}{\sc C.-C. Lee}, {\em On Robin problems with infinite-point BCs},  Bull. Lond. Math. Soc. {\bf 57}  (2025), no. 4, 1093--1117.

\bibitem{L2025}{\sc C.-C. Lee}, {\em An existence theorem for elliptic equations with nonlocal boundary conditions},
Bull. Sci. Math. {\bf 207} (2026), 103768.

\bibitem{LM2023}{\sc C.-C. Lee, M. Mizuno, S.-H. Moon}, {\em On the uniqueness of linear convection--diffusion equations with integral boundary conditions}, C. R. Math. Acad. Sci. Paris {\bf 361} (2023) 191--206.

\bibitem{L1837}
{\sc J. Liouville}, 
{\em Solution nouvelle d'un probl\`eme d'analyse, relatif aux ph\'enom\`enes thermom\'ecaniques}, 
J. Math. Pures Appl. \textbf{2} (1837) 439--456.

\bibitem{MF2013}{\sc D. Mantzavinos,  A.S. Fokas}, {\em The unified method for the heat equation: I. Non-separable boundary conditions and non-local constraints in one dimension}, European J. Appl. Math. 24 (2013), no. 6, 857--886.

%\bibitem{M1987}{\sc G.H. Meyer}, {\em Continuous orthonormalization for multipoint problems for linear ordinary differential equations}, SIAM J. Numer. Anal. {\bf 24} (1987) 1288--1300.






\bibitem{N2006}{\sc S.K. Ntouyas}, {\em Nonlocal initial and boundary value problems: a survey.} In Handbook of differential
equations: ordinary differential equations. Vol. {\bf 2} (2006), 461--557. North-Holland. 

%\bibitem{O1968} {\sc R.E. O'Malley Jr.}, {\em Topics in singular perturbations}, Adv. Math. {\bf 2} (1968) 365--470.

%\bibitem{O1991} {\sc R.E. O'Malley Jr.}, {\em Singular Perturbation Methods for Ordinary Differential Equations}, Springer Verlag, 1991.

\bibitem{PS2018}{\sc B. Pelloni , D.A. Smith},  {\em Nonlocal and multipoint boundary value problems for linear evolution equations}, Stud. Appl. Math. {\bf 141} (2018), 46--88.

%\bibitem{PW1984}{\sc M.H. Protter,  H.F. Weinberger}, {\em Maximum principles in differential equations}, Corrected reprint of the 1967 original. Springer-Verlag, New York (1984). 

%\bibitem{R1974}{\sc R.D. Russell}, {\em Collocation for systems of boundary value problems}, Numer. Math. {\bf 23} (1974), 119--133.

%\bibitem{S1998}{\sc P Souplet}, {\em Blow-up in nonlocal reaction–diffusion equations}, SIAM J. Math. Anal., {\bf 29} (1998) 1301--1334.

%\bibitem{S1993}{\sc R. Sperb},  {\em Optimal bounds in semilinear elliptic problems with nonlinear boundary conditions}, Z. Angew. Math. Phys. {\bf 44} (1993) 639--653. 

\bibitem{Sk1997} {\sc A.L. Skubachevskii}  (1997) {\em Elliptic Functional Differential Equations and Applications} (Vol. 91) Birkh\"{a}user.

%\bibitem{TD2012}{\sc M. Tatari, M. Dehghan}, {\em An efficient method for solving multi-point boundary value problems and applications in physics}, J. Vib. Control {\bf 18} (2012) 1116--1124.


%\bibitem{W1959}{\sc A.D. Wentzell}, {\em On boundary conditions for multi-dimensional diffusion processes}, Theory Probab. Appl. {\bf 4} (1959) 164--177.

%\bibitem{U1966}{\sc M. Urabe}, {\em Numerical solution of multi-point boundary value problems in Chebyshev series. Theory of the method}, Numer. Math. {\bf 9} (1966/67) 341--366.

%\bibitem{WZ2009}{\sc J.R.L. Webb, M. Zima}, {\em Multiple positive solutions of resonant and non-resonant nonlocal boundary value problems}, Non. Analysis {\bf 71} (2009) 1369--1378.

%\bibitem{Z1968}{\sc A. Zettl}, {\em Adjoint and self-adjoint boundary value problems with interface conditions}, SIAM J. Appl. Math. {\bf 16} (1968), 851--859.

%\bibitem{ZHZ2007}{\sc Y. Zou, Q. Hu, R. Zhang}, {\em On numerical studies of multi-point boundary value problem and its fold bifurcation}, Appl. Math. Comput. {\bf 185} (2007) 527--537. 

\bibitem{W2011}
{\sc A.M. Wazwaz}, {\em Fredholm Integro-Differential Equations.} In: Linear and Nonlinear Integral Equations. Springer, Berlin, Heidelberg (2011).

\bibitem{W1939}
{\sc H. Weyl}, {\em On the volume of tubes}, \textit{Amer. J. Math.}, {\bf 61} (1939) no. 3, pp. 461--472.
\end{thebibliography}
\end{document}